\documentclass[letterpaper]{amsart}
\usepackage{amsmath}
\usepackage{amstext}
\usepackage{amssymb}
\usepackage{amsfonts}
\usepackage{enumerate}
\usepackage{mathrsfs}
\usepackage{braket}
\usepackage{color}
\usepackage[all]{xy}
\usepackage{url}
\usepackage{bbm}
\usepackage{mathtools}
\usepackage[hidelinks]{hyperref}

\numberwithin{equation}{section}

\newtheoremstyle{note}% <name>
{1em}% <Space above>
{1em}% <Space below>
{}% <Body font>
{}% <Indent amount>
{\bfseries}% <Theorem head font>
{: }% <Punctuation after theorem head>
{.5em}% <Space after theorem headi>
{}% <Theorem head spec (can be left empty, meaning `normal')>

\newtheorem{theorem}{Theorem}[section]
\newtheorem{lemma}[theorem]{Lemma}
\newtheorem{proposition}[theorem]{Proposition}
\newtheorem{corollary}[theorem]{Corollary}

\theoremstyle{note}
\newtheorem{remark}[theorem]{Remark}

\newtheorem{definition}[theorem]{Definition}

\newtheorem*{acknowledgments}{Acknowledgments}

\newtheorem{claim}[theorem]{Claim}

\newcommand{\N}{{\mathbb{N}}}
\newcommand{\R}{{\mathbb{R}}}
\newcommand{\C}{{\mathbb{C}}}

\newcommand{\tn}[1]{{\left\vert\kern-0.25ex\left\vert\kern-0.25ex\left\vert #1 
    \right\vert\kern-0.25ex\right\vert\kern-0.25ex\right\vert}}

\newcommand{\OH}{\mathrm{OH}}
\newcommand{\eps}{\varepsilon}

\newcommand{\M}{{\mathrm{M}}}

\newcommand{\cU}{{\mathcal{U}}}

\newcommand{\cB}{{\mathcal{B}}}
\newcommand{\cH}{{\mathcal{H}}}
\newcommand{\cZ}{{\mathcal{Z}}}
\newcommand{\cL}{{\mathcal{L}}}

\DeclareMathOperator{\cb}{cb}

\DeclareMathOperator{\MIN}{MIN}

\title[Coarse geometry of operator spaces]{Towards a theory of coarse geometry of operator spaces}

\author[B. M. Braga]{Bruno M. Braga}

\address[B. M. Braga]{University of Virginia, $141$ Cabell Drive, Kerchof Hall, P.O. Box $400137$, Charlottesville, USA}
\email{demendoncabraga@gmail.com}
\urladdr{https: //sites.google.com/site/demendoncabraga/}


\begin{document}
\maketitle

\begin{abstract}
We introduce two notions of coarse embeddability between operator spaces:    \emph{almost complete coarse embeddability of bounded subsets} and \emph{spherically-complete coarse embeddability}.  We provide examples showing that these notions are strictly weaker than complete isomorphic embeddability --- in fact, they do not even imply isomorphic embeddability. On the other hand, we show that, despite their nonlinearity,  the existence of such  embeddings provides restrictions on the linear operator space structures of the spaces. Examples of nonlinear \emph{equivalences} between operator spaces are also provided.
\end{abstract}

 %\tableofcontents

\section{Introduction}\label{SectionIntro}

This paper continues the investigation of the nonlinear theory of operator spaces which was initiated in \cite{BragaChavezDominguez2020PAMS} and continued in \cite{BragaChavezDominguezSinclair}.   \emph{Operator spaces} are Banach subspaces of the space of bounded operators on some (complex) Hilbert space $H$, denoted by $\cB(H)$; so operator space theory is often regarded as noncommutative Banach space theory. The difference between the category of Banach spaces and the one of operator spaces does not lie in their objects, but   in the morphisms between them: for  operator spaces, one must consider   \emph{completely bounded operators}. Recall, given a set $X$ and $n\in\N$, $\M_n(X)$ denotes the set of $n$-by-$n$ matrices with entries in $X$ and, if  $f\colon X\to Y$ is a map into another set $Y$,   then the  \emph{$n$-th amplification of $f$} is the map $f_n\colon \M_n(X)\to \M_n(Y)$ given by 
\[f\big([x_{ij}]_{i,j=1}^n\big)=[f(x_{ij})]_{i,j=1}^n\]
for all $[x_{ij}]_{i,j=1}^n\in \M_n(X)$. If both $X$ and $Y$ are operator spaces, then each of $\M_n(X)$ and $\M_n(Y)$ comes equipped with a norm given by the canonical isomorphism   $\M_n(\cB(H))\cong \cB(H^{\oplus n})$. Therefore, if $f$ is linear, so is  each amplification $f_n$   and its operator norm is denoted by $\|f_n\|_n$. The \emph{complete bounded norm} of $f$ is then defined as 
\[\|f\|_{\cb}=\sup_{n\in\N}\|f_n\|_n.\] 
The linear map  $f$ is   called \emph{completely bounded} if $\|f\|_{\cb}<\infty$ and a \emph{complete isomorphic embedding} if both $f$ and $f^{-1}\colon f(X)\to X$ are completely bounded; if, furthermore, $f(X)=Y$, $f$ is a \emph{complete isomorphism}.

The first step  towards developing an ``interesting'' theory of nonlinear geometry of operator spaces is
to identify nontrivial types of nonlinear morphisms between operator spaces. We start by recalling that the most naive approach does not work. A map $f\colon X\to Y$ between operator spaces is said to be  \emph{completely coarse} if the sequence of amplifications $(f_n\colon \M_n(X)\to \M_n(Y))_n$ is \emph{equi-coarse}, i.e., if for all $r>0$ there is $s>0$ such  that 
\[\big\|[x_{ij}]_{i,j=1}^n-[z_{ij}]_{i,j=1}^n\big\|_{\M_n(X)}\leq r\Rightarrow \big\|f\big([x_{ij}]_{i,j=1}^n\big)-f\big([z_{ij}]_{i,j=1}^n\big)\big\|_{\M_n(Y)}\leq s\]
for all $n\in\N$ and all $[x_{ij}]_{i,j=1}^n, [z_{ij}]_{i,j=1}^n\in \M_n(X)$ --- for $n=1$, this is precisely the definition of  a \emph{coarse map} between Banach spaces.  As shown in \cite{BragaChavezDominguez2020PAMS}, the study of completely coarse maps between operator spaces does not lead to an interesting theory. Precisely:

  \begin{theorem}\emph{(}\cite[Theorem 1.1]{BragaChavezDominguez2020PAMS}\emph{)}
Let $X$ and $Y$ be  operator spaces, and let $f\colon  X \to Y$ be
completely coarse. If $f(0)= 0$, then $f$ is $\R$-linear.
\label{ThmCompleteCoarseIsRLinear}
\end{theorem}

One way to overcome this issue is, instead of considering a single function  with equi-coarse amplifications, to consider a sequence $(f^n\colon X\to Y)_n$ whose  amplifications $(f^n_n\colon \M_n(X)\to\M_n( Y))_n$ are equi-coarse --- i.e., each $f^n$ does not need to be completely coarse, but only to ``behave well'' until its $n$-th amplification. Consider the following \emph{linear} notion  (see  \cite[Definition 4.1]{BragaChavezDominguez2020PAMS}):   $X$ \emph{almost completely  isomorphically embeds into $Y$} if there is a   sequence $(f^n\colon X\to Y)_n$ of linear maps whose  amplifications $(f^n_n\colon \M_n(X)\to \M_n(Y))_n$ are \emph{equi-isomorphic embeddings}, i.e., there is $D\geq 1$ such that 
\[D^{-1}\big\|[x_{ij}]_{i,j=1}^n\big\|_{\M_n(X)}\leq \big\|f^n_n\big([x_{ij}]_{i,j=1}^n\big)\big\|_{\M_n(Y)}\leq D\big\|[x_{ij}]_{i,j=1}^n\big\|_{\M_n(X)}\]
for all  $n\in\N$ and all $[x_{ij}]_{i,j=1}^n\in \M_n(X)$. This   notion is  clearly weaker than complete isomorphic embeddability and, as shown in \cite[Theorem 4.2]{BragaChavezDominguez2020PAMS}, it is actually   \emph{strictly} weaker. %(see  also  Remark \ref{RemarkPrevPaper}).
  Moreover, this   has a natural nonlinearization: $X$ \emph{almost completely coarsely embeds into $Y$} if there is a   sequence $(f^n\colon X\to Y)_n$ whose amplifications $(f^n_n\colon \M_n(X)\to \M_n(Y))_n$ are \emph{equi-coarse embeddings}, i.e.,  $(f^n_n)_n$ is both equi-coarse and \emph{equi-expanding}:  for each $s>0$ there is $r>0$ such that \[\big\|[x_{ij}]_{i,j=1}^n-[z_{ij}]_{i,j=1}^n\big\|_{\M_n(X)}\geq  r \Rightarrow \big\|f\big([x_{ij}]_{i,j=1}^n\big)-f\big([z_{ij}]_{i,j=1}^n\big)\big\|_{\M_n(Y)}\geq  s\]
for all $n\in\N$ and all $[x_{ij}]_{i,j=1}^n, [z_{ij}]_{i,j=1}^n\in \M_n(X)$  (we refer to Subsection \ref{SubsectionCoarseNotions} for details on coarse geometry).

 As mentioend in \cite{BragaChavezDominguezSinclair},  for a nonlinear notion of embeddability between operator spaces to be relevant, at the bare minimum, the following must hold:
\begin{enumerate}[(I)]
\item\label{Item1} the nonlinear embedding  must   still be strong enough  to   recover some linear aspects of the operator space structures of the spaces, and 
\item\label{Item2}   the nonlinear embedding   must be strictly weaker than (almost) complete coarse embeddability.
\end{enumerate}
The results of \cite{BragaChavezDominguez2020PAMS,BragaChavezDominguezSinclair}   show that almost complete coarse  embeddability  satisfies both \eqref{Item1} and \eqref{Item2}.\footnote{Furthermore, \cite{BragaChavezDominguezSinclair} also   shows that almost complete \emph{Lipschitz} embeddability  satisfies both \eqref{Item1} and \eqref{Item2} above. }  As a sample example for \eqref{Item1}:  it was shown in  \cite[Theorem 1.2]{BragaChavezDominguez2020PAMS} that if an infinite dimensional  operator space $X$ almost completely coarsely embeds into G. Pisier's operator Hilbert space $\OH$, then $X$ must be  completely isomorphic to $\OH$.

As for  \eqref{Item2} above, although \cite[Theorem 1.4]{BragaChavezDominguezSinclair} provides  examples of operator spaces $X$ and $Y$ such that $X$ almost completely coarsely embeds into $Y$ but $X$ does  not almost completely  isomorphically  embed  into $Y$, the examples obtained therein  have some weaknesses:

\begin{enumerate}[(I)]\setcounter{enumi}{2}
\item\label{ItemI} the operator spaces $X$ and $Y$ are \emph{nonseparable}, and
\item\label{ItemII} the methods therein do not give  examples of nonlinear \emph{equivalences} between operator spaces, but only of \emph{embeddings}.
\end{enumerate}
In this paper, we propose an even weaker notion of nonlinear embeddability between operator spaces which also satisfies \eqref{Item1} and \eqref{Item2}, and, moreover,  provides  solutions to \eqref{ItemI}   and \eqref{ItemII}.

We now describe the main results of this paper. We start by  introducing the main kind of nonlinear embeddability  between operator spaces considered in this paper. Throughout, given a Banach space $X$, $B_X$ denotes its closed unit ball.

 \begin{definition} \label{DefiEmbIntro}
 Let $X$ and $Y$ be operator spaces.  We say  that the \emph{bounded subsets of $X$ almost completely coarsely embed into $Y$} if there is a sequence  of maps $(f^n\colon n\cdot B_X\to Y)_n$ such that the amplifications 
 \[\big(f^n_n\restriction_{ n\cdot B_{\M_n(X)}} \colon n\cdot B_{\M_n(X)}\to\M_n(Y)\big)_{n=1}^\infty\]
 are equi-coarse  embeddings.\footnote{Notice that $n\cdot B_{\M_n(X)}\subset \M_n(n\cdot B_X)$, so the restriction of each $f^n_n$ to $ n\cdot B_{\M_n(X)}$ is well defined.}
 \end{definition}

Almost complete coarse embeddability of bounded subsets is a transitive notion  for operator spaces. This notion is evidently  weaker than almost complete coarse embeddability. However, we show that the existence of such embeddings still captures some aspects of  the (linear) operator space structures of the spaces. Moreover, this can be seen even for Hilbertian  operator spaces --- recall, an operator space is \emph{Hilbertian} if it is isomorphic (as a Banach space) to a Hilbert space. 

Precisely, in order to state our first result, we recall the definition of the  row and column operator spaces. For each $i,j\in\N$,  $e_{i,j}$ denotes the operator in $\cB(\ell_2)$ whose $(i,j)$-coordinate is 1 and all others are 0. Then, the \emph{row operator space} and the \emph{column operator space} are defined by \[R=\overline{\mathrm{span}}\big\{e_{1,j}\mid j\in\N\big\}\text{ and }C=\overline{\mathrm{span}}\big\{e_{j,1}\mid j\in\N\big\},\]
  respectively.  Moreover, given $\theta\in [0,1]$, $(R,C)_\theta$ denotes the \emph{$\theta$-interpolation operator space of $R$ and $C$} (see Subsection \ref{SubsectionInterpolation} for details); so $(R,C)_\theta$  is also Hibertian. We prove  the following in Section \ref{SectionNotEmbedding}:

\begin{theorem}\label{ThmRandCnotEmbInterpolation}
Let $R$ and $C$ denote the column and row operator spaces, respectively, and let $\theta,\gamma \in [0,1]$ be distinct. Then, the bounded subsets of the interpolation operator space   $(R,C)_\theta$ do not almost  completely coarsely embed into  $(R,C)_\gamma$.
\end{theorem}

Since $(R,C)_0=R$, $(R,C)_1=C$ and $(R,C)_{1/2}=\OH$ completely isometrically, we have the following immediate corollary (the statement below can be considerably strengthened for $R$ and $C$, see  Remark \ref{RemarkUncollapsed}).

\begin{corollary}\label{CorRandCnotEmbInterpolation}
 The  operator spaces $R$, $C$, and $\OH$ are incomparable with respect to   almost  complete coarse  embeddability of bounded subsets.\qed
\end{corollary}

We point out that, in the Banach space setting, it was open for a long time whether there were Banach spaces which were incomparable with respect to coarse equivalence. This was only solved in \cite[Section 5]{BaudierLancienSchlumprecht2018JAMS}, where the authors showed that $\ell_2$ and the original Tsirelson space $T^*$ are incomparable.

We actually have   a much stronger result for nonlinear embeddability into $\OH$. Precisely,  using    methods of \cite{BragaChavezDominguez2020PAMS}, we obtain the following  strengthening of   \cite[Theorem 1.2]{BragaChavezDominguez2020PAMS}:

\begin{corollary}\label{CorOH}
If the bounded subsets of an infinite dimensional operator space $X$ almost completely coarsely embed into $\OH$,   then $X$ is completely isomorphic to $\OH$.
\end{corollary}

While Section \ref{SectionNotEmbedding} provides restrictions for the existence of almost complete coarse embeddings between bounded subsets of operators spaces, in Section \ref{SectionSomeEmb}, we provide nontrivial examples of when such embeddings exist. Moreover, our methods allow us to provide examples which are not only embeddings, but actually equivalences. For that, we consider the following:

 \begin{definition}\label{DefiEquivIntro}
 Let $X$ and $Y$ be operator spaces. We say   that the   \emph{bounded subsets of $X$ and $Y$ are  almost completely coarsely equivalent} if there is a sequence  of bijections   $(f^n\colon X\to Y)_n$ such  that, letting  $g^n=(f^n)^{-1}$ for each $n\in\N$,  the amplifications 
 \[\big(f^n_n \restriction_{ n\cdot B_{\M_n(X)}}\colon n\cdot B_{\M_n(X)}\to\M_n(Y)\big)_n\text{ and }\big(g^n_n \restriction_{ n\cdot B_{\M_n(Y)}} \colon n\cdot B_{\M_n(Y)}\to\M_n(X)\big)_n\]
 are equi-coarse embeddings.  
 \end{definition}

The next result shows that  the  new notion of nonlinear embeddability/equivalence  (Definitions \ref{DefiEmbIntro} and \ref{DefiEquivIntro}) overcomes      issues \eqref{ItemI} and \eqref{ItemII}. For that,   we adapt a method of N. Kalton   of finding nonisomorphic coarsely equivalent Banach spaces (\cite{Kalton2012MathAnnalen}) to the operator space setting. 

 \begin{theorem}\label{ThmSepBoundEquivSpacesEQUIV2SEPARABLE}
There are separable  operator spaces $X$ and $Y$ so that
\begin{enumerate}
\item $X$ does not  isomorphically embed into $Y$, and 
\item the bounded subsets of     $X$ and $Y$ are almost completely coarsely equivalent. 
\end{enumerate}  
\end{theorem}

Notice that the restriction for the linear embeddability of $X$ into $Y$ in Theorem \ref{ThmSepBoundEquivSpacesEQUIV2SEPARABLE} is a \emph{Banach space} restriction. To better understand the nuances of this new notion of nonlinear embeddability/equivalence, it would be interesting to have restrictions which are truly operator space ones. For that, we show that $R\cap C$ does not completely isomorphically embed into the $\ell_1$-sum of completely isomorphic copies of $\ell_\infty$  (see Propositon \ref{PropCompEmbSum}) and then, merging the methods above  with the well-known consequence of the noncommutative Khintchine inequality which gives a complete isomorphic embedding $R+C\to L_1[0,1]$ (see the discussion preceding  Corollary \ref{CorKaltonMethod} for details), we obtain the following:

  \begin{theorem}\label{ThmSepBoundEquivSpacesEQUIV2}
There are  operator spaces $X$ and $Y$ so that
\begin{enumerate}
\item $X$ linearly isomorphically embeds into $Y$, 
\item   $X$ does not completely isomorphically embed into $Y$,
and 
 \item the bounded subsets of     $X$ and $Y$ are almost completely coarsely equivalent.
\end{enumerate}  
\end{theorem}
  
 We point out that we do not know if the examples of $X$ and $Y$ provided by Theorem \ref{ThmSepBoundEquivSpacesEQUIV2} are such that $X$ does not \emph{almost} completely isomorphically embed into $Y$. Also, the examples in Theorem \ref{ThmSepBoundEquivSpacesEQUIV2} are nonseparable.

At last, in  Section \ref{SectionSphericalEmb},  we investigate a yet  different kind of nonlinear mapping between operator spaces.  Both in  classic operator space theory and in the nonlinear approach presented above, one considers  maps $f\colon X\to Y$ between operator spaces (or subsets of operator spaces) and then imposes restrictions on the  amplifications $f_n\colon \M_n(X)\to \M_n(Y)$. However,  the opposite approach is also possible:  one can start with a map   $F\colon \M_n(X)\to \M_n(Y)$ and then weaken the condition  of $F$ being the amplification of a map $X\to Y$. This  gives rise to what we call    \emph{spherical amplifications}. Throughout this paper, given a Banach space $X$, $\partial B_X$ denotes its unit sphere.

\begin{definition}\label{Defi.Spherical}
Let $X$ and $Y$ be operator spaces.
\begin{enumerate}
\item\label{Defi.Spherical.Item1} Given  $n\in\N$, a map $F\colon \M_n(X)\to \M_n(Y)$  is  a \emph{spherical amplification}  if for all $r>0$ there is $f\colon X\to Y$ such  that  \[F\restriction_{ r\cdot \partial B_{\M_n(X)}}=f_n \restriction_{ r\cdot \partial B_{\M_n(X)}}.\] 
\item\label{Defi.Spherical.Item2} A sequence $(F^n\colon \M_n(X)\to \M_n(Y))_n$ of spherical amplifications is a \emph{spherically-complete coarse embedding} if $(F^n)_n$ is a sequence of equi-coarse embeddings.
\end{enumerate}
\end{definition}

In other words, spherical amplifications are maps $\M_n(X)\to \M_n(Y)$ whose restrictions to the spheres centered at zero are amplifications. Notice that the composition of spherical amplifications    does not need to be a spherical amplification. So, it is not clear whether spherically-complete coarse embeddability is a transitive notion between operator spaces.  This notion does however sit in between   almost complete coarse embeddability and almost complete coarse embeddability of bounded subsets  (see Proposition \ref{PropSpheImpliesBounded}).  As such, the existence (and non existence) of such embeddings sheds  light on what extra properties nonlinear embeddings are allowed to have. 

As spherically-complete coarse embeddability is stronger than almost complete coarse embeddability of bounded subsets (Proposition \ref{PropSpheImpliesBounded}), the following is an immediate  corollary of Theorem \ref{ThmRandCnotEmbInterpolation}:

 \begin{corollary}
Let $R$ and $C$ be the column and row operator spaces, respectively, and  $\theta,\gamma \in [0,1]$ be distinct. Then, the interpolation operator space    $(R,C)_\theta$ does not spherically-completely  coarsely  embed into  $(R,C)_\gamma$.\qed
\end{corollary}

Obtaining examples of spherically-complete coarse embeddings is much harded than almost complete coarse embeddings of bounded subsets. One of the obstructions for that comes from the fact that the former requires functions which are defined and coarse embeddings on entire spaces; not only on bounded subsets.  However, using gluing techniques of N. Kalton (\cite{Kalton2012MathAnnalen}), we improve the construction of the maps on bounded subsets $n\cdot B_X\subset X\to Y$ given in  Theorem \ref{ThmSepBoundEquivSpacesEQUIV2SEPARABLE} in order to obtain spherical-amplifications defined on the whole $\M_n(X)$. As a final product, we obtain the following:

\begin{theorem}\label{ThmSepSphericalEmb}
There are separable operator spaces $X$ and $Y$ so that

\begin{enumerate}
\item $X$ does not  isomorphically embed into $Y$, and
\item    $X$ spherically-completely coarsely embeds into  $Y$. 
\end{enumerate}  
\end{theorem}

\iffalse
  \begin{figure*}[h] \centerline{%
 \xymatrix@R-2ex{ 
          \txt{Complete isomorphic embeddability }  \ar[d]\\  \txt{Almost complete isomorphic embeddability }  \ar[d]\\ \txt{Almost complete coarse embeddability } \ar[d] \\
        \txt{Spherically-complete coarse embeddability }\ar[d]\\  \txt{Almost complete coarse embeddability  of bounded subsets}   }
}
\caption{  Relation between different kinds of embeddability. }\label{Fig4}
\end{figure*}
\fi

 \section{Preliminaries}\label{SectionPrelim}
 
 \subsection{Basics on operator spaces}\label{SubsectionBasics}  For a detailed treatment of operator spaces, we refer to \cite{Pisier-OS-book}. 
 
 As described in the introduction, if $X$ is a subspace of $\cB(H)$, for some Hilbert space $H$, and $n\in\N$, then  $\M_n(X)$  comes equipped with the norm  induced by $\cB(H^{\oplus n})$ via the canonical inclusion $\M_n(X)\subset \M_n(\cB(H))$ and the canonical algebraic isomorphism $\M_n(\cB(H))=\cB(H^{\oplus n})$. We denote this norm by  $\|\cdot \|_{\M_n(X)}$.

 An arbitrary subset $X\subset \cB(H)$ (i.e., not necessarily a closed linear subset), where $H$ is a Hilbert space, is called an \emph{operator metric space}. Besides the metric that $X$ inherits   from $\cB(H)$, each $\M_n(X)$ inherits   a canonical metric given  from its containment in the operator space $\M_n(\cB(H))$. Elements $\bar x\in \M_n(X)$ have a canonical representation as $\bar x=[x_{ij}]_{i,j=1}^n$, where $x_{ij}\in X$ for all $i,j\in \{1,\ldots, n\}$. Moreover, for simplicity, we omit the indices and  simply write $[x_{ij}]$ for $[x_{ij}]_{i,j=1}^n$.

 If $X$ and $Y$ are operator spaces, we write $X\equiv Y$ to denote that $X$ and $Y$ are completely linearly isometric to each other. We denote the \emph{opposite operator space of $X$} by $X^\mathrm{op}$, i.e., $X^\mathrm{op}=X$ as a set and, for each $n\in\N$, the norm in $\M_n(X^\mathrm{op})$ is given by
 \[\|[x_{ij}]\|_{\M_n(X^\mathrm{op})}=\|[x_{ji}]\|_{\M_n(X )} \] for all $[x_{ij}]\in \M_{n}(X^\mathrm{op})$ (\cite[Section 2.10]{Pisier-OS-book}).

% A linear map $Q\colon X\to Y$ between Banach spaces is a \emph{quotient map} if it maps the open unit ball of $X$ onto the one of $Y$. If $X$ and $Y$ are moreover operator spaces, then $Q$ is called a \emph{complete quotient map} if each of its amplifications is a quotient map. In particular, a complete quotient is completely contractive.  A \emph{section} of $Q$ is a map $f\colon X\to Y$ so that $Q\circ f=\mathrm{Id}_X$. 

 \subsection{Interpolation spaces}\label{SubsectionInterpolation}
 
 Given $\theta\in [0,1]$, operator spaces $X$ and $Y$, and  homeomorphic linear embeddings of $X$ and $Y$ into  a topological vector space $Z$, we denote their \emph{complex $\theta$-interpolation space}   by $(X,Y)_\theta$. We refer the reader to the  monograph  \cite{BerghLofstromBook} for a detailed treatment of interpolation of Banach spaces and to   \cite[Section 2.7]{Pisier-OS-book} for its operator space version. Here, we only mention that the operator space structure of $(X,Y)_\theta$ is given by 
 \[\M_n((X,Y)_\theta)=(\M_n(X),\M_n(Y))_\theta \text{ for all } n\in\N.\] 
 
Notice that there is an abuse of notation in $(X,Y)_\theta$. Indeed, the ambient space $Z$ and the embeddings $X\hookrightarrow Z$ and $Y\hookrightarrow Z$ play a fundamental role on the space $(X,Y)_\theta$. Therefore, since we will consider $(R,C)_\theta$ and $(C,R)_\theta$ below, it is important to fix $Z$ and  embeddings $R\hookrightarrow Z$ and $C\hookrightarrow Z$. We do such in the standard way:   let $Z=\ell_2$ and let $R\hookrightarrow \ell_2$ and $C\hookrightarrow \ell_2$ denote the canonical  linear isometries. Equivalently, $(R,C)_\theta\equiv (R,R^\mathrm{op})_\theta$  and $(C,R)_\theta\equiv (C,C^\mathrm{op})_\theta$, where $  (R,R^\mathrm{op})_\theta$ and $  (C,C^\mathrm{op})_\theta$ are defined considering the identity maps $R\hookrightarrow R$ and $R^\mathrm{op}\hookrightarrow R$, and $C\hookrightarrow C$ and  $C^\mathrm{op}\hookrightarrow C$, respectively.

 The operator Hilbert space of G. Pisier, $\mathrm{OH}$, can be defined as a interpolation space of $R$ and $C$. Precisely, we have that $\OH\equiv (R,C)_{1/2}$ (see \cite[Corollary 7.11]{Pisier-OS-book}).

\subsection{Metric coarse notions}\label{SubsectionCoarseNotions}
We recall the basics of coarse geometry in this subsection.  Let $(X,d)$ and $(Y,\partial)$ be metric spaces, and $f\colon X\to Y$ be a map. The moduli $\omega_f,\rho_f\colon [0,\infty)\to [0,\infty]$ are defined by letting 
\[\omega_f(t)=\sup\{\partial(f(x),f(y))\mid d(x,y)\leq t\}\]
for all $t\geq 0$ and  
\[\rho_f(t)=\inf\{\partial(f(x),f(y))\mid d(x,y)\geq t\}\]
for all $t\geq 0$ (where the infimum over the empty set is taken  to be $\infty$). The map $f$ is \emph{coarse} if $\omega_f(t)<\infty$ for all $t\geq 0$ and \emph{coarse Lipschitz}   if there is $L>0$ such that $\omega_f(t)\leq Lt+L$ for all $t\geq 0$. The map $f$ is \emph{expanding} if $\lim_{t\to\infty}\rho_f(t)=\infty$ and a \emph{coarse embedding} if it is both coarse and expanding. If $f$ is coarse Lipschitz and, furthermore, there is $L>0$ such that $\rho_f(t)\geq L^{-1}t-L$ for all $t\geq 0$, then $f$ is a \emph{coarse Lipschitz embedding}.

The nonlinear embeddings between operator spaces  in this paper are  defined in terms of   sequences of maps between operator metric space. Therefore, we are interested in saying that a certain sequence is not only coarse/expanding, but it is so in a uniform way.  Precisely,  let $(X_n,d_n)_n$ and $(Y_n,\partial_n)_n$ be sequences of metric spaces and let   $\bar f=(f^n\colon X_n\to Y_n)_n$ be a sequence of maps. Then, we define moduli $\omega_{\bar f},\rho_{\bar f}\colon [0,\infty)\to [0,\infty]$ by letting 
\[\omega_{\bar f}(t)=\sup_{n}\omega_{f_n}(t)\text{ and }\rho_{\bar f}(t)=\inf_{n}\rho_{f_n}(t)\]
for all $t\geq 0$. Then, the sequence   $\bar f$ is \emph{equi-coarse} if $\omega_{\bar f}(t)<\infty$ for all $t\geq 0$, and $\bar f$ is  \emph{equi-coarse Lispchitz} if there is $L>0$ such  that $\omega_{\bar f}(t)\leq Lt+L$ for all $t\geq 0$. We say $\bar f$ is \emph{equi-expanding} if $\lim_{t\to \infty}\rho_{\bar f}(t)=\infty$. The sequence $\bar f$ is an \emph{equi-coarse embedding} if it is both equi-coarse and equi-expanding. If $\bar f$ is equi-coarse Lipschitz and there is $L>0$ such that $ \rho_{\bar f}(t)\geq L^{-1}t-L$ for all $t\geq 0$, then $\bar f$ is an \emph{equi-coarse Lipschitz embedding}.

\begin{remark}\label{RemarkAffineBound} If $X$ is a convex subspace of a Banach space and $Y$ is a Banach space, then   a coarse map $f \colon X\to Y$   is  automatically  coarse Lipschitz  (\cite[Lemma 1.4]{KaltonNonlinear2008}).\footnote{More generally, this holds as long as $X$ is a \emph{metrically convex} metric space in the sense of \cite[Page 16]{KaltonNonlinear2008} and $Y$ is an arbitrary metric space.} Precisely,   taking $L=\omega_f(1)$, we have that \[\|f(x)-f(y)\|\leq L\|x-y\|+L\] for all $x,y\in X$.  Similarly, if $\bar f=(f^n\colon X_n\to Y_n)_n$ is a sequence of equi-coarse embeddings between convex subsets of Banach spaces, then, letting   $L=\omega_{\bar f}(1)$, we have that 
\[\|f^n(x)-f^n(y)\|\leq L\|x-y\|+L\]
for all $n\in\N$ and all $x,y\in X$. On the other hand, we point out   that $\rho_f$  and $\rho_{\bar f}$ do not need to be bounded below by an affine map.
 \end{remark}

 The following is the main notion of embeddability studied in these notes: 
 
 \begin{definition}[Definition \ref{DefiEmbIntro}]
 Let $X$ and $Y$ be operator spaces. We say   that the \emph{bounded subsets of $X$ almost completely coarsely} (resp. \emph{coarse Lipschitzly}) \emph{embed into $Y$} if there is a sequence  of maps $(f^n\colon n\cdot B_X\to Y)_n$ such that the amplifications 
 \[\big(f^n_n\restriction_{ n\cdot B_{\M_n(X)}}\colon n\cdot B_{\M_n(X)}\to\M_n(Y)\big)_n\]
 are equi-coarse (resp. equi-coarse Lipschitz) embeddings. The sequence $(f^n\colon n\cdot B_X\to Y)_n$ is called an \emph{almost complete coarse} (resp. \emph{coarse Lipschitz}) \emph{embedding of bounded subsets}.
 \end{definition}

Despite its nonlinear nature and despite not even being defined on the whole $X$, an almost complete coarse embedding of bounded subsets induces a complete $\R$-isomorphic embedding into ultraproducts of the target space.  Precisely, the next proposition shows that the maps $(f^n\colon n\cdot B_X\to Y)_n$ can be combined  to obtain a map into $Y^\cU$ defined on the whole $X$.\footnote{Here $Y^\cU$ denotes the operator   space ultraproduct of $Y$ with respect to $\cU$.  We refer the reader to \cite[Subsection 2.8]{Pisier-OS-book} for definitions.}

\begin{proposition} \emph{(c.f.} \cite[Proposition 4.4]{BragaChavezDominguez2020PAMS}\emph{)}
Let $X$ and $Y$ be operator spaces, and suppose that the bounded subsets of  $X$   almost completely  coarsely  embed   into $Y$.  Then  $X$  completely $\R$-isomorphically embeds into  $Y^\cU$ for any nonprincipal ultrafilter $\cU$ on $\N$.\label{PropAlmSphUltra}
\end{proposition}

\begin{proof}
	Let $(f^n\colon n\cdot B_X \to Y )_n$ be such  that $\bar f=(f^n_n\colon n\cdot B_{\M_n(X)}\to \M_n(Y))_n$ is a sequence of equi-coarse embeddings and let  $\cU$ be a nonprincipal ultrafilter on $\N$. Then, let  $g\colon X\to Y^\cU$ be given by 
\[g(x)=[(f^{n}(x))_n] \text{ for all }x\in X,\] 
where $[(f^{n}(x))_n] $ represents   the equivalence class of $(f^{n}(x))_n$  in $Y^\cU$. Notice that, since for each $x\in X$,   $x\in n\cdot B_X$ for all $n\geq \|x\|$, the map $g$ is well defined.  Moreover,  it is straightforward to check that 
\[\rho_{\bar f}(\|\bar x-\bar y\|_{\M_n(X)})\leq \|g_n(\bar x)-g_n(\bar y)\|_{\M_n(Y^\cU)}\leq \omega_{\bar f}(\|\bar x-\bar y\|_{\M_n(X)})\]
for all $n\in \N$ and all $\bar x,\bar y\in \M_n(X)$. In particular,  $g$ is completely coarse and, by \cite[Theorem 1.1]{BragaChavezDominguez2020PAMS}, $g$ must be $\R$-linear. Therefore, it follows from the inequality above  that  $g$ must be a complete $\R$-isomorphic embedding.
\end{proof}

\begin{corollary}[Corollary \ref{CorOH}]
If the bounded subsets of an operator space $X$ almost completely coarsely embed into $\OH(J)$, for some index set $J$, then $X$ is completely isomorphic to $\OH(I)$, for some index set $I$.\footnote{Given a set $J$, $\OH(J)$ is the operator Hilbert space of G. Pisier with density character equal to the cardinality of $J$. See \cite[Chapter 7]{Pisier-OS-book} for details.}\label{CorOH2}
\end{corollary}

\begin{proof}
This proof follows exactly the proof of \cite[Theorem 1.2]{BragaChavezDominguez2020PAMS} but with Proposition \ref{PropAlmSphUltra} replacing \cite[Proposition 4.4]{BragaChavezDominguez2020PAMS}.
\end{proof}

As usual, the  notion of embeddability described above has its ``equivalence version". Precisely:

 \begin{definition}[Definition \ref{DefiEquivIntro}]
 Let $X$ and $Y$ be operator spaces. We say   that the   \emph{bounded subsets of $X$ and $Y$ are  almost completely coarsely} (resp. \emph{coarse Lipschitzly}) \emph{equivalent} if there is a sequence  of bijections   $(f^n\colon X\to Y)_n$ such  that, letting  $g^n=(f^n)^{-1}$ for each $n\in\N$,  the amplifications 
 \[\big(f^n_n\colon n\cdot B_{\M_n(X)}\to\M_n(Y)\big)_n\text{ and }\big(g^n_n\colon n\cdot B_{\M_n(Y)}\to\M_n(X)\big)_n\]
 are equi-coarse (resp. equi-coarse Lipschitz) embeddings. The sequence $(f^n\colon n\cdot B_X\to Y)_n$ is called an \emph{almost complete coarse} (resp. \emph{coarse Lipschitz}) \emph{equivalence of bounded subsets}.
 \end{definition}

 \section{Incompatibility  of the interpolations of $(R,C)$ }\label{SectionNotEmbedding}

In this section, we prove Theorem \ref{ThmRandCnotEmbInterpolation}. For didactic reasons, we choose to present the proof of a particular case of  Theorem \ref{ThmRandCnotEmbInterpolation} before its proof in full  generality. The anxious reader can safely skip the next proposition and the remark following it. 

\begin{proposition}\label{PropRandCnotEmb}
Let $R$ and $C$ denote the column and row operator spaces, respectively. Then   the bounded subsets of $R$ do not almost completely coarsely embed into  $C$ and vice versa.
\end{proposition}

Notice that Proposition \ref{PropRandCnotEmb} is   Theorem \ref{ThmRandCnotEmbInterpolation} for $\theta=0$ and $\gamma=1$. Indeed, this is the case since $(R,C)_0\equiv(C,R)_1\equiv R$  and $(C,R)_0\equiv (R,C)_1\equiv C$. 

\begin{proof}[Proof of Proposition \ref{PropRandCnotEmb}]
Suppose $(f^n\colon  n\cdot B_R\to C)_n$ is such that $\bar f=(f^n_n\colon  n\cdot B_{\M_n(R)}\to \M_n(C))_n$ is a sequence of equi-coarse embeddings. As $\lim_{t\to \infty} \rho_{\bar f}(t)=\infty$, pick  $r>0$ such that $\rho_{\bar f}(r)>0$. 

  For each $n\in\N$ and $j\in \{1 ,\ldots, n\}$,  let $a_{n,j}\in \M_{n}(R)$ be the element whose $(j,1)$-coordinate is $re_{1,2j-1}$ and all other coordinates are zero, and let $b_{n,j} \in  \M_{n}(R)$ be the element whose $(j,1)$-coordinate is $re_{1,2j}$ and all other coordinates are zero; so $\|a_{n,j}\|_{\M_{n}(R)}=\|b_{n,j}\|_{\M_{n}(R)}=r$. It is also clear that   $\|a_{n,j}-b_{n,j}\|_{\M_{n}(R)}=\sqrt{2} r$, so it follows that 
\begin{equation}\label{EqRhoR}
\|f^{n}_n(a_{n,j})-f^{n}_n(b_{n,j})\|_{\M_{n}(C)}\geq \rho_{\bar f}(\sqrt{2}r)>0
 \end{equation}
for all $n\geq r$ and all $j\in \{1,\ldots, n\}$. Hence, since
 \[f^{n}_n(a_{n,j})-f^{n}_n(b_{n,j})=\begin{bmatrix}
 0&0&\ldots&0\\
  \vdots&\vdots  &\ddots &\vdots \\
   0&0&\ldots&0\\
f^{n}(re_{1,2j-1})-f^{n}(re_{1,2j})&0&\ldots &0\\
 
 0&0&\ldots&0\\
 \vdots&\vdots  &\ddots &\vdots \\
  0&0&\ldots&0\\
\end{bmatrix} \] 
and $f^{n}(re_{1,2j-1})-f^{n}(re_{1,2j})\in C$ for each $n\geq r $ and each $j\in\{1,\ldots, n\}$, it is clear from the operator space structure of $C$ and \eqref{EqRhoR} that   \[\lim_n\Big\|\sum_{j=1}^n\big(f_n^{n}(a_{n,j})-f_n^{n}(b_{n,j})\big) \Big\|_{\M_{n}(C)}=\infty.\]

For each $n\in\N$, let $c_n,d_n\in \M_{n}(R)$ be given by $c_n=\sum_{j=1}^na_{n,j}$ and $d_n=\sum_{j=1}^nb_{n,j}$; so $\|c_n\|_{\M_{n}(R)}=\|d_n\|_{\M_{n}(R)}=r$ and $\|c_n-d_n\|_{\M_{2n}(R)}=\sqrt{2}r$. Then,

\begin{align*}  
\Big\|\sum_{j=1}^n\big(f_n^{n}(a_{n,j})-& f_n^{n}(b_{n,j})\big)\Big\|_{\M_{n}(C)}\\
&=
\left\|\begin{bmatrix}
f^{n}(re_{1,1})-f^{n}(re_{1,2})&0&\ldots &0\\
f^{n}(re_{1,3})-f^{n}(re_{1,4})&0&\ldots &0\\
 \vdots&\vdots  &\ddots &\vdots \\
f^{n}(re_{1,2n-1})-f^{n}(re_{1,2n})&0&\ldots &0\\
\end{bmatrix}\right\|_{\M_{n}(C)}\\
&= \|f^{n}_n(c_n)-f^{n}_n(d_n)\|_{\M_{n}(C)}\\
&\leq \omega_{\bar f}(\sqrt{2}r)
 \end{align*}
 for all $n\in\N$. As $\omega_{\bar f}(\sqrt{2}r)<\infty$, this gives a contradiction.
 
 The proof that the bounded subspaces of $C$ do  not almost completely coarsely  embed into $R$ follows similarly and we leave the details to the reader. 
\end{proof}

\begin{remark}\label{RemarkUncollapsed}
Notice that we do not use the full power of  almost  complete coarse embeddability of bounded subsets in Proposition \ref{PropRandCnotEmb}. Indeed,   it is clear from its proof that, beside the family  $(f^n_n\colon  n\cdot B_{\M_n(R)}\to \M_n(C))_n$ being  equi-coarse,  we only need the existence of a single $r>0$ for which $\rho_{\bar f}(r)>0$. Moreover, the proof of Proposition \ref{PropRandCnotEmb} gives us that there is no $(f^n\colon    B_{R}\to C)_n$  such that  $(f^n_n\colon    B_{\M_n(R)}\to \M_n(C))_n$  is equi-coarse and, for some $t\in [0, 1]$ and some $\delta>0$, we have that 
\[\|f^n_n(\bar x)-f^n_n(\bar y)\|>\delta\]
for all $\bar x,\bar y\in   B_{\M_n(R)}$ with $\|\bar x-\bar y\|=\sqrt{2} t$. The existence of such  family   is an operator space version of a   map  which is coarse and \emph{almost uncollapsed} in the sense of \cite[Section 1]{Braga2018JFA} (see also \cite[Section 1]{Rosendal2017ForumSigma} for  \emph{uncollapsed maps}).
\end{remark}

Before presenting the proof of Theorem \ref{ThmRandCnotEmbInterpolation}, we need a technical lemma.

\begin{lemma}\label{LemmaLowerBoundInterpolation}
Let $\theta>0$ and $n\in\N$. Then
\[ \left\|\begin{bmatrix}
e_{1,m_1}&0&\ldots &0\\
e_{1,m_2}&0&\ldots &0\\
\vdots &\vdots &\ddots&\vdots\\
e_{1,m_n}&0&\ldots &0
 \end{bmatrix}\right\|_{\M_n((R,R^\mathrm{op})_\theta)}= n^{\theta/2}.
\]  
for all $m_1<\ldots< m_n\in \N$.  In particular, if  $(x_m)_m$ is  a semi-normalized\footnote{Recall, a sequence $(x_n)_n$ in a normed space is \emph{semi-normalized} if it is bounded and $\inf_n\|x_n\|>0$.} weakly null sequence in $R$, then there is $D>0$ depending only on $\inf_m\|x_m\|$ and an infinite $M\subset \N$ such that, for all $m_1<m_2<\ldots<m_n\in M$, we have 
\[\left\|\begin{bmatrix}
x_{m_1}&0&\ldots& 0\\
x_{m_2}&0&\ldots& 0\\
\vdots&\vdots&\ddots &\vdots\\
x_{m_n}&0&\ldots & 0
\end{bmatrix}\right\|_{\M_n((R,R^\mathrm{op})_\theta)}\geq D n^{\theta/2}.\]
\end{lemma}

\begin{proof}
Without loss of generality, let $m_i=i$ for each $i\in \{1,\ldots,n\}$. Let $b\in  \M_n((R,R^\mathrm{op})_\theta)$ denote  the first matrix in  the statement of the lemma. Clearly,  $\|b\|_{\M_n(R)}=1$ and $\|b\|_{\M_n( R^\mathrm{op}) }=n^{1/2}$. Therefore,  since  \[\|b\|_{\M_n((R,R^\mathrm{op})_\theta)}\leq \|b\|^{1-\theta}_{\M_n(R)}\|b\|^{\theta}_{\M_n( R^\mathrm{op}) } \]  (see  \cite[Section 1.2.30]{BlecherLeMerdy2004}), we have that $\|b\|_{\M_n((R,R^\mathrm{op})_\theta)} \leq n^{\theta/2}$. So we are left to show the opposite inequality.

 Let $(e_j)_j$ be the standard unit basis of $\ell_2$. For each $j\in \{1,\ldots, n\}$, let $\xi_j=(\xi_{j}(i))_{i=1}^n\in \ell_2^{\oplus n}$ be the vector whose first coordinate is $e_j$ and all others are zero, and let  $\zeta_j=(\zeta_{j}(i))_{i=1}^n \in \ell_2^{\oplus n}$ be the vector whose $j$-th  coordinate is $e_1$  and all others are zero. Let  $f\colon \M_n(R)\to \C$ be given by $f(a)=\sum_{j=1}^n\langle a\xi_j, \zeta_j \rangle$ for all $a\in \M_n(R)$. So,     $f(b)=n$. The functional $f$ can be seen as an element in both $\M_n(R)^*$ and $\M_n(R^\mathrm{op})^*$. Considering $f$ as a functional on $\M_n(R)$, it is immediate that  $\|f\|_{\M_n(R)^*}\leq  n$ and, as $\|b\|_{\M_n(R)}=1$,   $\|f\|_{\M_n(R)^*}= n$. On the other hand, considering $f$ as a functional in $\M_n(R^{\mathrm{op}})^*$, an application of Cauchy-Schwarz inequality gives that $\|f\|_{\M_n(R^\mathrm{op})^*}\leq n^{1/2}$ and, as  $\|b\|_{\M_n(R^\mathrm{op})}=n^{1/2}$, we have that   $\|f\|_{\M_n(R^{\mathrm{op}})^*}=  n^{1/2}$. 
  
  By definition,   $\M_n((R,R^\mathrm{op})_\theta)=(\M_n(R),\M_n(R^\mathrm{op}))_\theta$ isometrically.  Therefore, since $\M_n(R)$ is reflexive (\cite[Theorem 2.5.1]{Pisier-OS-book}),  we have that \[(\M_n(R),\M_n(R^\mathrm{op}))_\theta^*\equiv(\M_n(R)^*,\M_n(R^\mathrm{op})^*)_\theta\] (\cite[Corollary 4.5.2]{BerghLofstromBook}). So,    \[\|f\|_{\M_n((R,R^\mathrm{op})_\theta)^*}\leq \|f\|_{\M_n(R )^*}^{1-\theta}\|f\|_{\M_n(R^{\mathrm{op}})^*}^\theta\leq  n^{1-\theta/2}\] (\cite[Theorem 4.1.2]{BerghLofstromBook}) and, as $f(b)=n$, this gives us that $\|b\|_{\M_n((R,R^\mathrm{op})_\theta)}\geq n^{\theta/2}$. 
  
 For the last statement, let  $(x_m)_m$ be a semi-normalized    weakly null sequence. By going to a subsequence if necessary, we can assume that $(x_m)_m$ is equivalent to $(e_{1,m})_m$. Let $E=\overline{\mathrm{span}}\{x_m\mid m\in\N\}$ as a Banach space, so the   linear operator $T\colon E\to R$ which sends each $x_m$ to $e_{1,m}$ is  an isomorphism. Let $E_1$ and $E_2$ denote   $E$  endowed with the operator space structure induced by the subset inclusions $E_1\subset R$ and   $E_2\subset R^{\mathrm{op}}$, respectively. Since $R$ is a homogeneous\footnote{An operator space $X$ is \emph{homogeneous} if $\|u\|_{\cb}=\|u\|$ for any linear operator $u:X\to X$.} Hilbertian space,  $E_1$ is completely isometric to $R$ (\cite[Proposition 9.2.1]{Pisier-OS-book}). Therefore, $E_1$ is also homogeneous, so $T\colon E_1\to R$ is a complete isomorphism. Analogously,   $T\colon E_2\to R^{\mathrm{op}}$ is also a complete isomorphism.  Therefore,  $T\colon (E_1,E_2)_\theta\to (R,R^{\mathrm{op}})_\theta$ is a complete isomorphism. The result then follows since $\|b\|\geq n^{\theta/2}$.
\end{proof}

\begin{proof}[Proof of Theorem \ref{ThmRandCnotEmbInterpolation}]
Suppose $\theta<\gamma$.  To simplify notation, let $X=  (R,R^{\mathrm{op}})_\theta $ and $Y= (R,R^{\mathrm{op}})_\gamma $. Suppose $(f^n\colon n\cdot B_X\to Y)_n$ is an almost complete coarse embedding of bounded subsets, so $\bar f=(f^n_n\colon n\cdot B_{\M_n(X)}\to \M_n(Y))_n$ is an equi-coarse embedding. Pick $r>0$ such that $\rho_{\bar f}(r)>0$ and, to  simplify notation, for each $n\in\N$, let $\alpha_n=rn^{\theta/2}$ and  $g^n=f^{\alpha_n}$.     

Fix $n\in\N$. By Rosenthal's $\ell_1$-theorem (see \cite[The Main Theorem]{Rosenthal1974PNAS}), going to a subsequence if necessary, we can assume that $(g^n(re_{1,j}))_j$ is weakly Cauchy, so $(g^n(re_{1,2j-1})-g^n(re_{1,2j}))_j$ is weakly null. Moreover, $(g^n(re_{1,2j-1})-g^n(re_{1,2j}))_j$  is semi-normalized. Indeed, since $\| e_{1,2j-1}- e_{1,2j}\|_X= \sqrt{2}$, we have that 
  \begin{align*}
  \| 
g^n(re_{1,2j-1})-g^n(re_{1,2j}) \|_{Y}  
  \geq \rho_{\bar f}(\|re_{1,2j-1}- re_{1,2j}\|_{X}) =\rho_{\bar f}(\sqrt{2}r).
\end{align*}
So, $(g^n(re_{1,2j-1})-g^n(re_{1,2j}))_j$ is semi-normalized and its lower bound does not depend on $n$. Therefore,  going to a subsequence if necessary,  there is $D>0$ (which does not depend on $n$) such that 
\begin{align*}
  \left\|\begin{bmatrix}
g^n(re_{1,1})-g^n(re_{1,2})&0&\ldots &0\\
g^n(re_{1,3})-g^n(re_{1,4})&0&\ldots &0\\ 
\vdots& \vdots &\ddots&\vdots\\
g^n(re_{1,2n-1})-g^n(re_{1,2n})&0&\ldots &0
 \end{bmatrix}\right\|_{\M_{n}(Y)}\geq Dn^{\gamma/2}.
\end{align*}
  (see Lemma \ref{LemmaLowerBoundInterpolation}).

For  each $j\in \{1,\ldots,n\}$,  let $a_{n,j}\in \M_{n}(X)$  be the operator whose $(j,1)$-coordinate is $ re_{1,2j-1}$  and all other coordinates are zero, and let $b_{n,j}\in \M_{n}(X)$  be the operator whose $(j,1)$-coordinate is $ re_{1,2j}$  and all other coordinates are zero; so,    $\|a_{n,j}\|_{\M_n(X)}=\|b_{n,j}\|_{\M_n(X)}=r$. Therefore,   
\begin{align*} \Big\|\sum_{j=1}^n\big(g^{n}_n(a_{n,j})-&g^{n}_n(b_{n,j})\big) \Big\|_{\M_{n}(Y)}\\
&=  \left\|\begin{bmatrix}
g^n(re_{1,1})-g^n(re_{1,2})&0&\ldots &0\\
g^n(re_{1,3})-g^n(re_{1,4})&0&\ldots &0\\ 
\vdots& \vdots &\ddots&\vdots\\
g^n(re_{1,2n-1})-g^n(re_{1,2n})&0&\ldots &0
 \end{bmatrix}\right\|_{\M_{n}(Y)}\\
 &\geq D n^{\gamma/2}.
 \end{align*}

Let $c_n\in \M_n(X)$ be the operator whose  $(j,1)$-coordinate is $re_{1,2j-1}$, for all $j\in \{1,\ldots, n\}$,   and all other coordinates are zero, and let $d_n\in \M_n(X)$ be the operator whose  $(j,1)$-coordinate is $re_{1,2j}$, for all $j\in \{1,\ldots, n\}$,   and all other coordinates are zero. By Lemma \ref{LemmaLowerBoundInterpolation},   $\|c_n\|_{\M_n(X)}=\|d_n\|_{\M_n(X)}=rn^{\theta/2}$.  In particular, $\|c_n-d_n\|_{\M_n(X)}\leq 2rn^{\theta/2}$.

Since  $\omega_{\bar f}$ is bounded above by an affine map (see Remark \ref{RemarkAffineBound}), pick   $L>0$ such that  $\omega_{\bar f}(t)=Lt+L$ for all $t >0$.  Then, as each $ c_n,d_n,a_{n,j},b_{n,j}\in \alpha_n\cdot   B_{\M_{n}(X)}$, we have 
\begin{align*}  
\Big\|\sum_{j=1}^n\big(g^{n}_n(a_{n,j})-& g^{n}_n(b_{n,j})\big)\Big\|_{\M_{n}(Y)}\\
&=
\left\|\begin{bmatrix}
g^{n}(re_{1,1})-g^{n}(re_{1,2})&0&\ldots &0\\
g^{n}(re_{1,3})-g^{n}(re_{1,4})&0&\ldots &0\\
 \vdots&\vdots  &\ddots &\vdots \\
g^{n}(re_{1,2n-1})-g^{n}(re_{1,2n})&0&\ldots &0\\
\end{bmatrix}\right\|_{\M_{n}(Y)}\\
&= \|g^{n}_n(c_n)-g^{n}_n(d_n)\|_{\M_{n}(Y)}\\
&\leq \omega_{\bar f}\big(\|c_n-d_n\|_{\M_{n}(X)}\big)\\
&\leq 2L   r n^{\theta/2}+L.
 \end{align*}

As $n\in\N$ was arbitrary, we conclude that  \[Dn^{\gamma/2}\leq 2L   r n^{\theta/2}+L\] for all $n\in\N$; contradiction, since $\theta<\gamma$. So, the bounded subsets of $(R,R^{\mathrm{op}})_\theta$ do  not almost  completely coarsely embed into  $(R,R^{\mathrm{op}})_\gamma$.

If $\theta>\gamma$, then  $1-\theta<1-\gamma$ and   an analogous argument gives us that the bounded subsets of $(R^{\mathrm{op}},R)_{1-\theta}$ do  not almost   completely coarsely embed into  $(R^{\mathrm{op}},R)_{1-\gamma}$. Therefore, as   $(R,R^{\mathrm{op}})_\theta\equiv(R^{\mathrm{op}},R)_{1-\theta}$ and  $(R,R^{\mathrm{op}})_\gamma\equiv(R^{\mathrm{op}},R)_{1-\gamma}$ (see \cite[Theorem 4.2.1]{BerghLofstromBook}), we are done.
\end{proof}

\section{Almost complete nonlinear equivalences of bounded subsets}\label{SectionSomeEmb}
 
In \cite{Kalton2012MathAnnalen}, N. Kalton presented a remarkable method of constructing nonisomorphic   Banach spaces which are coarsely/uniformly equivalent to each other. In a nutshell, his method consists of the following: given a quotient map $Q\colon Y\to X$ between Banach spaces, N. Kalton  constructs a Banach space $\cZ(Q)$ and a quotient map $\tilde Q\colon \cZ(Q)\to X$ such that $\cZ(Q)$   is isomorphic to the $\ell_1$-sum of isomorphic copies of $Y$ and  $\cZ(Q)$ is coarsely equivalent to $X\oplus \ker(\tilde Q)$ (\cite[Section 8]{Kalton2012MathAnnalen}). Choosing   appropriate quotient maps   $Q\colon Y\to X$, this method gives us many interesting examples. For instance: 
\begin{enumerate}
\item If $Q$ is a quotient of $\ell_1$ onto $c_0$, then it is immediate that  $c_0 $ does not isomorphically embeds into $\cZ(Q)$. Indeed, being an $\ell_1$-sum of Schur spaces,\footnote{Recall, a Banach space $X$ is \emph{Schur}, or has the \emph{Schur property}, if every weakly convergent sequence in $X$  is  norm convergent.} $\cZ(Q)$ is also Schur, but $c_0$ is not.
\item Moreover, if $Q$ is a quotient of $\ell_1$ onto $c_0$, then $c_0\oplus \ker(\tilde Q)$ is not uniformly equivalent\footnote{Two metric spaces $X$ and $Y$ are \emph{uniformly equivalent} if there is a bijection $f\colon X\to Y$ such that both $f$ and $f^{-1}$ are uniformly continuous.} to $\cZ(Q)$ (\cite[Theorem 8.9]{Kalton2012MathAnnalen}).\footnote{\cite[Theorem 8.9]{Kalton2012MathAnnalen} provides the first example of coarsely equivalent Banach spaces which are not uniformly equivalent. N. Kalton's result is actually much stronger:  $c_0$ does not even coarsely Lipschitzly embed into $\cZ(Q)$ by a uniformly continuous map.}
\item There exist   separable $\mathcal L_1$-subspaces  of $\ell_1$ which are uniformly equivalent   but not linearly isomorphic (\cite[Theorem 8.6]{Kalton2012MathAnnalen}).
\end{enumerate}
In this section, we   revisit this  method in the context of operator spaces and   use it in order to  obtain  Theorem \ref{ThmSepBoundEquivSpacesEQUIV2SEPARABLE} and \ref{ThmSepBoundEquivSpacesEQUIV2}.

Let $X$ and $Y$ be operator spaces. Following \cite{Kalton2012MathAnnalen}, we denote by $\cH(X,Y)$ the space of all \emph{positively homogeneous} functions $f\colon X\to Y$, i.e., \[f(\alpha x)=\alpha f(x) \text{ for all }   x\in X \text{ and all }\alpha\geq 0,\] such that \[\|f\|\coloneqq\big\{|f(x)|\mid x\in B_X\big\}<\infty.\]
Given $\eps>0$ and $f\in \cH(X,Y)$, we denote the infimum of all $L>0$ such that 
\[\|f(x)-f(y)\|\leq L\max\big\{\|x-y\|,\eps\|x\|,\eps\|y\|\big\}\text{ for all }x,y\in X\]
by   $\|f \|^\eps$.\footnote{This was used in the context of Banach space theory in \cite{Kalton2012MathAnnalen} but with   notation $\|\cdot\|_\eps$. In order to avoid over usage of lower indices, we opted for an upper index here.}   One can check that $\|\cdot\|^\eps$ is a Banach norm on $\cH(X,Y)$. Moreover, notice that if $k\in\N$, then $f_k\in \cH(\M_k(X),\M_k(Y))$, so $\|f_k\|^\eps$ is well defined. 

Elements in $\cH(X,Y)$ do not need to be coarse. However, they are clearly so on bounded sets and we can actually obtain some quantitative estimates. Precisely, this is the content of our next lemma, which is  inspired by the proof of \cite[Lemma 7.4]{Kalton2012MathAnnalen} (see Lemma \ref{LemmaKaltonGluing} below for an actual version of  \cite[Lemma 7.4]{Kalton2012MathAnnalen} to our setting).
 
 \begin{lemma}\label{LemmaHomogenousCoarseEst}
Let $X$ and $Y$ be operator spaces, $k\in\N$, $t,K>0$, and let $f\in \cH(X,Y)$ be such that   $\|f_k\|^{e^{-t}}\leq K$. Then, for any $r,s>0$, we have that  
\[\|f_k(\bar x)-f_k(\bar y)\|_{\M_k(Y)}\leq K\|\bar x-\bar y\|_{\M_k(X)}+K(r+s)
\]
 for all $\bar x,\bar y\in (re^t+s)\cdot   B_{\M_k(X)}$.
\end{lemma}

\begin{proof} 
Let $\bar x,\bar y\in (re^t+s)\cdot   B_{\M_k(X)}$ and suppose $\|\bar x\|\geq \|\bar y\|$. Then, since $\|f_k\|^{e^{- t}}\leq K$ and $\|\bar x\|\leq re^t+s$, we have 
\begin{align*}
\|f_k(\bar x)-f_k(\bar y)\|_{\M_k(Y)}&\leq K\max\Big\{\|\bar x-\bar y\|_{\M_k(X)},e^{-t}\|\bar x\|_{\M_k(X)}\Big\}\\
&\leq  K \|\bar x-\bar y\|_{\M_k(X)}+K(s+r).
\end{align*}
\end{proof}

Recall,    a bounded linear map $Q\colon X\to Y$ between Banach spaces is a \emph{quotient map} if there is $\delta>0$ for which $Q(B_X)$ contains $\delta\cdot  B_Y$; if it is necessary to emphasize  $\delta$, we say $Q$ is a  \emph{$\delta$-quotient map}. In particular, $Q$ is surjective.  If $X$ and $Y$ are operator spaces, in order not to cause confusion between quotient and \emph{complete} quotient maps,\footnote{A completely bounded map $Q\colon Y\to X$ between operator spaces is a \emph{complete quotient} if there is $\delta>0$ such that each of its amplifications is a  $\delta$-quotient.} we refer to quotient maps   $ Y\to X$ as \emph{Banach} quotient maps.

Consider operator spaces $X$ and $Y$, and a completely bounded map $Q\colon Y\to X$ which is also a Banach quotient. For each $m\in\N$,  we define a norm  $\|\cdot\|_{Y_{m}}$ on  $Y$   by letting 
\[\|y\|_{Y_m}=\max\big\{2^{-m}\|y\|,\|Q(y)\|\big\}   \text{ for all } x\in Y.\]
We denote by $Y_m$ the Banach space consisting of $Y$ endowed with the equivalent norm $\|\cdot\|_{Y_m}$. Moreover,  we endow $Y_m$ with the operator space structure given by  
\[\|[y_{ij}]\|_{\M_k(Y_m)}=\max\Big\{2^{-m}\|[y_{ij}]\|_{\M_k(Y)},\|[Q(y_{ij})]\|_{\M_k(X)}\Big\}, \] for all $ k\in\N$, and all $[y_{ij}]\in\M_k( Y)$. It follows straightforwardly from Ruan's theorem that these norms induce an operator space structure on $Y_m$ (see \cite[Section 2.2]{Pisier-OS-book}). Moreover, each $Y_m$  is completely isomorphic to $Y$ (but not isometric).  

We define \[\cZ(Q\colon Y\to X)=\big(\bigoplus_n Y_m\big)_{\ell_1},\] i.e., $\cZ(Q\colon Y\to X)$ is the $\ell_1$-sum of $(Y_m)_m$ (see \cite[Section 2.6]{Pisier-OS-book} for definitions). For simplicity, we  simply write $\cZ(Q)=\cZ(Q\colon Y\to X)$. For each $m\in\N$, there is a canonical complete linear isometric embedding $i_m\colon Y_m\hookrightarrow  \cZ(Q)$. By the universal property of $\ell_1$-direct sums, there is a completely bounded  map   \[\tilde Q\colon \cZ(Q)\to X\]  such that $\tilde Q\circ i_m=Q$ for all $m\in\N$ (see \cite[Subsection 1.4.13]{BlecherLeMerdy2004}). Clearly, $\tilde Q$ is also a Banach quotient map.

Recall, given a Banach  quotient map $Q\colon Y\to X$, a map $f\colon X\to Y$ is called a  \emph{section of $Q$} if $Q\circ f=\mathrm{Id}_X$.

\begin{lemma}\label{LemmaSections}
Let  $Q\colon Y\to X$ be a completely bounded map which is a Banach quotient  and let $\tilde Q\colon \cZ(Q)\to X$ be given as above. Then,   for each $k\in\N$ and each $\eps>0$, $\tilde Q$ admits  a section $f\in \cH(X,\cZ(Q))$ such that $\|f_k\|^\eps\leq 1$.
\end{lemma}

\begin{proof}
Fix $k\in\N$ and $\eps>0$. As $Q$ is a quotient map, there is $C>0$ and an assignment      $f\colon  \partial B_X\to    Y$ such that $\|f(x)\|\leq C$ and $Q(f(x))=x$ for all $x\in \partial B_X$. We extend $f$ to the whole $X$ by letting $f(0)=0$ and $f(x)=\|x\|f(x/\|x\|)$ for all $x\in X\setminus \{0\}$. So $f$ is a section of $Q$ in $\cH(X,Y)$. Pick $m\in \N$ such that $ 2^{-m+1}Ck^2\leq\eps$. Then, if $[x_{ij}],[y_{ij}]\in \M_k(X)$, we have that 
\begin{align*}
\|&[f(x_{ij})-  f(y_{ij})]\|_{\M_k(Y_m)}\\
&=\max\{2^{-m}\|[f(x_{ij})-f(y_{ij})]\|_{\M_k(Y)},\|[x_{ij}-y_{ij}]\|_{\M_k(X)}\}\\
&\leq \max\{2^{-m+1}\|[f(x_{ij})]\|_{\M_k(Y)},2^{-m+1}\|[f(y_{ij})]\|_{\M_k(Y)}, \|[x_{ij}-y_{ij}]\|_{\M_k(X)}\}\\
&\leq \max\{2^{-m+1}Ck^2\|[x_{ij}]\|_{\M_k(X)},2^{-m+1}Ck^2\|[y_{ij}]\|_{\M_k(X)},\|[x_{ij}-y_{ij}]\|_{\M_k(X)}\}\\
&\leq \max\{\eps\|[x_{ij}]\|_{\M_k(X)},\eps\|[y_{ij}]\|_{\M_k(X)},\|[x_{ij}-y_{ij}]\|_{\M_k(X)}\}.
\end{align*}
Therefore, considering $f$ as a map $X\to Y_m$ and considering the canonical (completely isometric) inclusion $Y_m\hookrightarrow \cZ(Q)$, we obtain that the map $f\colon X\to \cZ(Q)$ satisfies $\|f_k\|^\eps\leq 1$.  
\end{proof}

We   now gather the previous results and present a (nontrivial) method to obtain operator spaces which have   almost completely coarse Lipschitzly equivalent  bounded subsets (cf. \cite[Proposition 8.4]{Kalton2012MathAnnalen}).

\begin{theorem}\label{ThmKaltonMethod}
Let $X$ and $Y$ be operator spaces,  $Q\colon Y\to X$ be a completely bounded map which is also a Banach quotient, and let $\tilde Q\colon \cZ(Q)\to X$ be as above.   Then, the bounded subsets of  $\cZ(Q)$  and $X\oplus \ker(\tilde Q)$ are  almost completely coarse Lipschitzly equivalent.
\end{theorem}

\begin{proof}
To simplify notation, let $Z=\ker(\tilde Q)$. For each $k\in\N$, let $f^{k}\colon X\to \cZ(Q)$  be the positively homogeneous section of $\tilde Q\colon  \cZ(Q)\to X$ given by Lemma \ref{LemmaSections} for $\eps=e^{-	k}$, so $\|f^k_k\|^{e^{-k}}\leq 1$. For each $k\in\N$,  define maps $g^{k}\colon  \cZ(Q)\to X\oplus Z$ and $h^{k}\colon  X\oplus Z\to  \cZ(Q)$ by letting \[g^{k}(  y)=\big(\tilde Q(  y),  y-f^{k}(\tilde Q(  y))\big)\text{ and }h^{k}(  x,  z)=  z+f^{k}(  x)\] for all $  y\in  \cZ(Q) $ and all $( x,  z)\in X\oplus Z$. Notice that, as each $f^{k}$ is a section of $\tilde Q$, then $g^{k}$ and $h^{k}$ are inverses of each other for all $k\in\N$. 

By Lemma \ref{LemmaHomogenousCoarseEst}, 
\begin{equation*}\label{EqCoarseEqui1}
\|f_k^{k }(\bar x)-f_k^{k }(\bar x')\|_{\M_k(Y)}\leq \|\bar x-\bar x'\|_{\M_k(X)}+1
\end{equation*} 
for all $k\in \N$  and all $\bar x,\bar x'\in e^k\cdot B_{\M_k(X)}$. Then, as $\tilde Q$ is completely contractive, the previous inequality and $g^k$'s formula imply that  
\[\|g_k^{k}(\bar y)-g_k^{k}(\bar z)\|_{\M_k(X\oplus Z)}\leq 2\|\bar y-\bar z\|_{\M_k(\cZ(Q))}+1\]
for all $k\in \N$  and all $\bar y,\bar z\in e^k\cdot B_{\M_k(\cZ(Q))}$. In particular, since $g^{k}(0)=0$, this implies that \[g_k^{k}\Big(e^k\cdot B_{\M_k(\cZ(Q))}\Big)\subset (2e^k+1)\cdot B_{\M_k(X\oplus Z)}\]
for all $k\in\N$. Therefore, using Lemma \ref{LemmaHomogenousCoarseEst} again, it is clear from $h^k$'s formula that 
\[\|h^{k}_k( \bar x, \bar  z)-h^{k}_k( \bar  x',\bar   z')\|_{\M_k(\cZ(Q))}\leq 2\|(\bar x,\bar z)-(\bar x',\bar z')\|_{\M_k(X\oplus Z)}+3\]
for all $k\in \N$ and all $(\bar x,\bar z),(\bar x',\bar z')\in (2e^k+1)\cdot B_{\M_k(X\oplus Z)}$. As $g^{k}$ and $h^{k}$ are inverses of each other, this implies that 
\[\|g_k^{k}(\bar y)-g_k^{k}(\bar z)\|_{\M_k(X\oplus Z)}\geq \frac{1}{2}\|\bar y-\bar z\|_{\M_k(\cZ(Q))}-\frac{3}{2}\]
for all $\bar y,\bar z\in e^k\cdot B_{\M_k(\cZ(Q))}$. Therefore, 

\[\Big(g_k^{k}\restriction_{ e^k\cdot B_{\M_k(\cZ(Q))}}\colon e^k\cdot B_{\M_k(\cZ(Q))}\to \M_{k}(X\oplus Z)\Big)_{k} \]
are  equi-coarse Lipschitz embeddings. Analogous calculations show that 
\[\Big(h_k^{k}\restriction_{e^k\cdot B_{\M_k(X\oplus Z)}}\colon e^k\cdot B_{\M_k(X\oplus Z)}\to \M_{k}(\cZ(Q))\Big)_{k} \] are also 
   equi-coarse Lipschitz embeddings. So,   we are done.
\end{proof}

\subsection{Applications of   Theorem  \ref{ThmKaltonMethod}}
Before the next corollary, we make a trivial observation: it is well known that every separable Banach space is a quotient of $\ell_1$ (\cite[Theorem 2.3.1]{AlbiacKaltonBook}). Hence, given any Banach space $X$, there is a completely bounded Banach quotient map $\MIN(\ell_1)\to \MIN(X)$. 

 The next corollary  shows that the notion of   almost complete coarse embeddability of bounded subsets overcomes issues \eqref{ItemI} and \eqref{ItemII}  from Section \ref{SectionIntro}.

\begin{corollary}[Theorem \ref{ThmSepBoundEquivSpacesEQUIV2SEPARABLE}]
Let $Q\colon \MIN(\ell_1)\to\MIN( c_0)$ be a completely bounded  map which is also a Banach quotient, and let $\tilde Q\colon \cZ(Q)\to X$ be as above.  Then  
\begin{enumerate}
\item\label{CorKaltonMethodMIN.Item2} $\MIN( c_0) $ does not  isomorphically embed  into  $\cZ(Q)$, and 
\item \label{CorKaltonMethodMIN.Item1} the bounded subsets of $\cZ(Q)$ and  $\MIN( c_0)\oplus \ker(\tilde Q)$  are  almost completely coarse Lispchitzly equivalent.
\end{enumerate}  \label{CorKaltonMethodMIN}
\end{corollary}

\begin{proof}
  \eqref{CorKaltonMethodMIN.Item1} follows from   Theorem  \ref{ThmKaltonMethod}. For \eqref{CorKaltonMethodMIN.Item2}, notice that since $\cZ(Q)$ is the $\ell_1$-sum of Schur spaces,  $\cZ(Q)$ is also Schur (see \cite[Lemma 8.1(i)]{Kalton2012MathAnnalen}). Hence, as $\MIN(c_0)$ is not Schur, $\MIN(c_0)$ does not   isomorphically embed into  $\cZ(Q)$.
\end{proof}

 \begin{remark}
Notice that Corollary \ref{CorKaltonMethodMIN} provides a way to obtain Theorem \ref{ThmSepBoundEquivSpacesEQUIV2SEPARABLE}, but many other quotient maps would do the same. For instance, Theorem \ref{ThmSepBoundEquivSpacesEQUIV2SEPARABLE} can also be obtained using the fact that every separable operator space $X$ is a complete quotient of $(\bigoplus_n\M_n^*)_{\ell_1}$ (see \cite[Proposition  1.12.2]{Pisier-OS-book}). 
 \end{remark}

The remainder of this subsection is dedicated to prove Theorem \ref{ThmSepBoundEquivSpacesEQUIV2}.  Recall,   the operator space $R\cap C$ is the   operator subspace  of the $\ell_\infty$-sum $R\oplus C$ defined as the  image of the map \[x\in \ell_2\mapsto (u(x),v(x)) \in R\oplus C,\]
where $u\colon \ell_2\to R$ and $v\colon \ell_2\to C$ are the canonical linear isometries.  In particular, $R\cap C$ is also a Hilbertian operator space.

The next result pertains to the (complete) isomorphic theory of operator spaces. It was obtained jointly with   Timur Oikhberg and the author is grateful for   Oikhberg's kind permission to include it here.

\begin{proposition}\label{PropCompEmbSum}
Let $(Y_m)_m$ be a sequence of operator spaces all of which  are isomorphic to $\ell_\infty$. Then, the operator space $R\cap C$ does not completely isomorphically embed into $(\bigoplus_n Y_n)_{\ell_1}$.
\end{proposition}

\begin{proof}
Suppose for a contradiction that there is a complete isomorphic embedding $u\colon R\cap C\to (\bigoplus_n Y_n)_{\ell_1}$. Let $i\colon R\cap C\to R$ be the canonical isometry, so $i$ is   a complete contraction. As $R$ is an injective operator space (\cite[Theorem 4.5]{Ruan1989Trans}), the completely bounded map $i\circ u^{-1}\colon u(R\cap C)\to R$ extends to a completely bounded map $v\colon (\bigoplus_n Y_n)_{\ell_1}\to R$. Therefore, $p=u\circ i^{-1}\circ v$ is a bounded projection of $(\bigoplus_n Y_n)_{\ell_1}$ onto $u(R\cap C)$. From now on, the proof follows entirely in the Banach space level.

 For each $m\in\N$, let $p_m\colon (\bigoplus_n Y_n)_{\ell_1}\to (\bigoplus_{n=1}^mY_n)_{\ell_1}$ be the canonical projection.

\begin{claim}\label{ClaimPropCompEmb} 
For all $\eps>0$ there is $m\in \N$ and an infinite dimensional subspace $X\subset R\cap C$ such that $\|\mathrm{Id}_{u(X)}-p_m\restriction_{u(X)}\|<\eps$.
\end{claim}

\begin{proof}
As $R\cap C$ is isomorphic to $\ell_2$ and $(\bigoplus_n Y_n)_{\ell_1}$ is an $\ell_1$-sum,  this follows by a standard gliding hump argument from Banach space theory and we leave the details to the reader.
\end{proof}

As $u(R\cap C)$ is isomorphic to $\ell_2$, there is $C>1$ such that every subspace of $u(R\cap C)$ is the image of a projection on $u(R\cap C)$ of norm at most $C$. Fix a positive $\eps$ smaller than $\min\{1/(C\|p\|),1/(2\|u^{-1}\|)\}$, and let $m
\in \N$ and $X\subset R\cap C$ be given by Claim \ref{ClaimPropCompEmb} for $\eps$.   By our choice of $C$, $u(X)$ is the image of a projection $p'\colon u(R\cap C)\to u(X)$ of norm at most $C$. Therefore,   letting $q=p'\circ p$, we have that  $(q-q\circ p_m)\restriction_{u(X)}$  is an operator in $\cL(u(X))$ --- the space of bounded operators on $u(X)$ --- and, as $\eps<1/(C\|p\|)$, this operator has  norm less than 1. As $q\restriction_{u(X)}=\mathrm{Id}_{u(X)}$, it follows from basic Banach algebra theory that   $w=q\circ p_m\restriction_{u(X)}$ is an invertible operator in $\cL(u(X))$. It now easily follows that  the operator $p_m\circ w^{-1}\circ q\restriction_{(\bigoplus_{n=1}^mY_n)_{\ell_1}}$ is a bounded projection of $(\bigoplus_{n=1}^mY_n)_{\ell_1}$ onto $p_m\circ u(X)$.

As $\eps<1/(2\|u^{-1}\|)$, it is immediate that $p_m\circ u$ is an isomorphic embedding of $X$ onto $p_m\circ u(X)$. So, $p_m\circ u(X)$ is isomorphic to $\ell_2$. As each $Y_n$ is isomorphic to $\ell_\infty$, the space $(\bigoplus_{n=1}^mY_n)_{\ell_1}$ is also isomorphic to $\ell_\infty$. Therefore, by the previous paragraph, $\ell_\infty$ contains a complemented subspace isomorphic to $\ell_2$. This is a contradiction since all complemented subspaces of $\ell_\infty$ are isomorphic to $\ell_\infty$  (\cite[Theorem 5.6.5]{AlbiacKaltonBook}).
\end{proof}

Before the last corollary, we must recall that  there is a complete quotient map $Q\colon L_\infty[0,1]\to R\cap C$.  Indeed, by   \cite[Theorem 9.8.3]{Pisier-OS-book}, there is a complete isomorphic embedding $F\colon R+ C\to L_1[0,1]$.\footnote{\label{Footnote} The space $R+C$ is the quotient operator space of the $\ell_1$-sum $R\oplus_1 C$ by $\Delta=\{(x,-^tx)\mid  x\in R\}$. We refer the reader to \cite[Page 194]{Pisier-OS-book} for further details} Therefore, since $R\cap C=(R+C)^*$ completely isometrically (see \cite[Page 194]{Pisier-OS-book}), the map $Q=F^*\colon L_\infty[0,1]\to R\cap C$ is the desired quotient. 

%Moreover, replacing $Q$ by $Q/(2\|Q\|)$, we can assume that $\|Q\|\leq 1/2$. Then, considering $(Y_m)_m$ as defines above for $Y=L_{\infty}[0,1]$, it follows that $Y_1$ is completely isometric to $L_\infty[0,1]$. 

 \begin{corollary}[Theorem   \ref{ThmSepBoundEquivSpacesEQUIV2}]  Let $Q\colon L_\infty[0,1]\to R\cap C$ be  the complete quotient map defined above. Then, we have that
 \begin{enumerate}
 \item \label{ItemCorKaltomMethod2}  $(R\cap C)\oplus \ker(\tilde Q)$ linearly isomorphically  embeds into $\cZ(  Q)$, 
 \item\label{ItemCorKaltomMethod1}  $R\cap C$ does not completely isomorphically embed into $\cZ(  Q)$, and 
  \item  the bounded subsets of $\cZ(Q)$ and $(R\cap C)\oplus \ker(\tilde Q)$ are almost  completely coarse Lipschitzly equivalent.
\end{enumerate}\label{CorKaltonMethod}
 \end{corollary}

\begin{proof}
The last statement is simply Theorem \ref{ThmKaltonMethod}. For \eqref{ItemCorKaltomMethod2}, notice that, as   $L_\infty[0,1]=L_\infty[0,1]\oplus L_\infty[0,1]$, $\cZ(Q)$ is linearly isomorphic to $(\bigoplus_n Z_n)_{\ell_1}$, where $Z_1=L_\infty[0,1]\oplus L_\infty[0,1]$ and $Z_n=Y_n$ for all $n>1$ (here $(Y_n)_n$ is given as in the definition of $\cZ(Q)$). Then, as $L_\infty[0,1]$ contains $R\cap C$ linearly isometrically and $(\bigoplus_n Y_n)_{\ell_1}$ contains  $\ker(\tilde Q)$ linearly isometrically, \eqref{ItemCorKaltomMethod2} follows. Item \eqref{ItemCorKaltomMethod1}  follows immediately from Proposition   \ref{PropCompEmbSum} and the fact that $L_\infty[0,1]$ is isomorphic to $\ell_\infty$ (\cite[Theorem 4.3.10]{AlbiacKaltonBook}).
\end{proof}

 \section{Spherical amplifications}\label{SectionSphericalEmb}

 In this section, instead of amplifying maps $X\to Y$ in order to obtain maps $\M_n(X)\to \M_n(Y)$, we follow the opposite road: we start with  maps  $F\colon \M_n(X)\to \M_n(Y)$ and then look at weakenings   of the property of $F$ being the amplification of a map $X\to Y$. This gives rise to what we call    \emph{spherical amplifications}, which in turn gives us \emph{spherically-complete coarse embeddings}.    
 
\begin{definition}[Definition \ref{Defi.Spherical}\eqref{Defi.Spherical.Item1}]
Let $X$ and $Y$ be operator spaces, $n\in\N$, and $F\colon \M_n(X)\to \M_n(Y)$ be a map. 

\begin{enumerate}
\item Given $A\subset \M_n(X)$, $F$ is called an \emph{amplification on $A$} if there is $f\colon X\to Y$ such that $F\restriction_{ A}=f_n \restriction_{A}$. 
\item The map $ F$ is called a \emph{spherical amplification}  if $F $ is an   amplification on $r\cdot \partial B_{\M_n(X)}$ for all $r\geq 0$. 
\end{enumerate}
\end{definition}

In other words, given operator spaces $X$ and $Y$, a map  $F\colon \M_n(X)\to \M_n(Y)$ is a spherical amplification if there is a  family of maps $(f^r\colon X\to Y)_{r\geq 0}$ such  that  \[F\restriction_{r\cdot \partial B_{\M_n(X)}}=f^r_n\restriction_{r\cdot \partial B_{\M_n(X)}}\] for all $r\geq 0$. Such family is called a \emph{witness that $F$ is a spherical amplification}.  Clearly, if $f\colon X\to Y$ is a map, then each  amplification $f_n\colon \M_n(X)\to\M_n(Y)$ is a spherical amplification.

Before investigating nonlinear embeddings given by spherical amplifications, we prove some simple properties of this new class of maps. We start by showing that witnesses that maps are spherical amplifications have the following uniqueness property for $n>1$: 

\begin{proposition}\label{PropUniquenessSphAmp}
Let $X$ and $Y$ be operator spaces, $n>1$, and $F\colon \M_n(X)\to \M_n(Y)$ be a spherical amplification. If $(f^r)_{r\geq 0}$ and $(g^r)_{r\geq 0}$ are witnesses that $F$ is a spherical amplification, then $f^r\restriction_{r\cdot B_X}=g^r\restriction_{ r\cdot B_X}$ for all $r\geq 0$.
\end{proposition}

Before proving this proposition, notice that this statement is clearly false if $n=1$. Indeed, for $n=1$, any function $F\colon \M_1(X)\to \M_1(Y)$ is an amplification of itself since $\M_1(X)=X$ and $\M_1(Y)=Y$.  So, letting $f^r=F$ for all $r\geq 0$, $(f^r)_{r\geq 0}$ witnesses that $F$ is a spherical amplification. On the other hand, we can define each $g^r\colon X\to Y$ to equal $F$ on $r\cdot \partial B_X$ and to equal anything else outside 
$r\cdot \partial B_X$; clearly,  $(g^r)_{r>0}$   also witnesses that $F$ is a spherical  amplification.

\begin{proof}[Proof of Proposition \ref{PropUniquenessSphAmp}]
Fix $r\geq 0$ and $x\in r\cdot B_X$. Pick an arbitrary $y\in \M_{n-1}(X)$ with norm $r$. Then, $\begin{bmatrix}
x&0\\
0& y 
\end{bmatrix}\in \M_n(X)$ has norm $r$ and we have that
\[\begin{bmatrix}
f^r(x)&f^r(0)\\
f^r(0)& f^r(y)
\end{bmatrix}
=F\left(\begin{bmatrix}
x&0\\
0& y
\end{bmatrix}\right)
=
\begin{bmatrix}
g^r(x)&g^r(0)\\
g^r(0)& g^r(y)
\end{bmatrix}.\]
Hence,  $f^r(x)=g^r(x)$ and the result follows.
\end{proof}

\begin{remark}\label{RemarkOutrBall}
Notice that the statement that   $(f^ r)_{r\geq 0}$ is  a witness that $F\colon \M_n(X)\to \M_n(Y)$ is  a spherical amplification is a statement which only asserts things about each $ f^ r$ on $r\cdot B_X$. Indeed, this follows since $\|[x_{ij}]\|_{\M_n(X)}\geq \|x_{\ell k}\|$ for all $[x_{ij}]\in \M_n(X)$ and all $\ell, k\in\{1,\ldots, n\}$. Therefore, we can modify each $f^ r$ outside $r\cdot B_X$ as we wish without losing the property of $(f^ r)_{r\geq 0}$  witnessing that  $F\colon \M_n(X)\to \M_n(Y)$ is  a spherical amplification.
\end{remark}

The next proposition shows that, in the linear setting, spherical amplifications are actual amplifications.  

\begin{proposition}
Let $X$ and $Y$ be operator spaces, $n\in\N$, and $F\colon \M_n(X)\to \M_n(Y)$ be a  spherical amplification. If $F$ is $\R$-linear, then $F$ is an   amplification. 
\end{proposition}

\begin{proof}
If $n=1$, any map is an amplification, so we assume $n>1$. Let $(f^r\colon X\to Y)_{r\geq 0}$ witness that $F$ is a spherical amplification.  Notice that $f^r(0)=0$ for all $r\geq 0$. Indeed, picking $x\in X$ and $y\in\M_{n-1}(X)$ both with norm $r$, we have that  
\begin{align*}
\begin{bmatrix}
f^r(x)&f^r(0)\\
f^r(0)&f^r (y)
\end{bmatrix}&=F\left(\begin{bmatrix}
 x&0\\
0&y
\end{bmatrix}\right)\\
&=F\left(\begin{bmatrix}
 x&0\\
0&0
\end{bmatrix}\right)+F\left(\begin{bmatrix}
 0&0\\
0&y
\end{bmatrix}\right) \\
&= \begin{bmatrix}
 f^r(x)+f^r(0)&f^r(0)+f^r(0)\\
f^r(0)+f^r(0)&f^r(0)+f^r(y)
\end{bmatrix} , \end{align*}
so $f^r(0)=0$ as claimed. 

For each $r>0$,   define 
\[g^r(x)=\left\{\begin{array}{ll}
f^r(x),& \text{ for }x\in r\cdot B_X,\\
\frac{\|x\|}{r}f^r\Big(r\frac{x}{\|x\|}\Big),& \text{ for  }x\not\in r\cdot B_{X}.
\end{array}\right.\]
By Remark \ref{RemarkOutrBall}, $(g^ r)_{r>0}$   witnesses that $F$ is a spherical amplification.   Clearly, we also have that $g^r(0)=0$ for all $r>0$; this will be used below with no further mention. 

\begin{claim}
$g^r(x)=  \frac{\|x\|}{r}g^r\Big(r\frac{x}{\|x\|}\Big)$ for all $r> 0$ and all $x\in X\setminus \{0\}$.
\end{claim}

\begin{proof}
Fix $r> 0 $ and $x\in X\setminus\{0\}$. If $x\not\in r\cdot B_X$, this follow from the definition of $g^r$. Say $\|x\|\leq r$. Then, picking $y\in \M_{n-1}(X)$ with norm $r$, we have  

\begin{align*}
\begin{bmatrix}
g^r(x)&0\\
0& g^r(y)
\end{bmatrix}
&=
F\left(\begin{bmatrix}
x&0\\
0& y
\end{bmatrix}
\right)\\
&=\frac{\|x\|}{r}F\left(\begin{bmatrix}
r\frac{x}{\|x\|}&0\\
0& 0 
\end{bmatrix}
\right)+F\left(\begin{bmatrix}
0&0\\
0& y
\end{bmatrix}\right)\\
&=\begin{bmatrix}
\frac{\|x\|}{r}g^r(r\frac{x}{\|x\|})&0\\
0& g^r(y)
\end{bmatrix}.
\end{align*}
So, $ g^r(x)=\frac{\|x\|}{r}g^r(r\frac{x}{\|x\|})$ and  the claim follows.
\end{proof}

\begin{claim}
$g^r=g^s$ for all $r,s>0$.
\end{claim}

\begin{proof}
Fix $s>r>0$.  If $x\in X$ has norm $r$, then, picking $y\in \M_{n-1}(Y)$ with norm $s$, we have 
\begin{align*}
\begin{bmatrix}
g^r(x)&0\\
0& 0
\end{bmatrix}
&=
F\left(\begin{bmatrix}
x&0\\
0& 0
\end{bmatrix}
\right)\\
&=F\left(\begin{bmatrix}
x&0\\
0& y
\end{bmatrix}
\right)-F\left(\begin{bmatrix}
0&0\\
0&  y
\end{bmatrix}\right)\\
&=\begin{bmatrix}
g^s(x)&0\\
0& 0
\end{bmatrix}.
\end{align*}
So, $g^r(x)=g^s(x)$. Therefore, by the previous claim, we have that, for an arbitrary $x\in X\setminus \{0\}$,
\[g^r(x)=\frac{\|x\|}{r}g^r\Big(r\frac{x}{\|x\|}\Big)=\frac{\|x\|}{r}g^s\Big(r\frac{x}{\|x\|}\Big)=g^s(x),\]
and the claim follows.\end{proof}

By the previous claim, it follows that $F=g^r_n$ for any $r>0$, so $F$ is an  amplification. 
\end{proof}

We now recall the notion of nonlinear embeddability between operator spaces   given by spherical amplifications:

\begin{definition}[Definition \ref{Defi.Spherical}\eqref{Defi.Spherical.Item2}]
Let $X$ and $Y$ be operator spaces, and $(F^n\colon \M_n(X)\to \M_n(Y))_n$ be a sequence of spherical amplifications.  If $(F^n)_n$ are equi-coarse (resp. equi-coarse Lipschitz) embeddings, then $(F^n)_n$ is called a \emph{spherically-complete coarse} (resp. \emph{coarse Lipschitz}) \emph{embedding}.
\label{DefiSphericalEmb}
\end{definition}

\begin{proposition}\label{PropSpheImpliesBounded}
Let $X$ and $Y$ be operator spaces.  If  $X$    spherically-completely  coarsely (resp. coarse Lipschitzly) embeds   into $Y$, then the bounded subsets of $X$ almost completely coarsely (resp. coarse Lipschitzly) embed into $Y$. 
\end{proposition}

\begin{proof}
  Let $(F^n\colon \M_n(X)\to \M_n(Y))_n$ be an almost spherically-complete coarse (resp. coarse Lipschitzly) embedding and let $\rho,\omega\colon [0,\infty)\to[0,\infty)$   witness that.  For each $n\in\N$, let $(f^{r,n}\colon X\to Y)_r$ witness that $F^n$ is a spherical amplification. For each $n\in\N$, let $g^n=f^{n,n+1}$ and $x_n\in X$ have norm $n$. Then 
\begin{align*}
\|g^n_n(\bar x)-g^n_n(\bar y)\|_{\M_n(Y)}&=\left\|\begin{bmatrix}
g^n_n(\bar x)-g^n_n(\bar y)&0\\
0&0
\end{bmatrix}\right\|_{\M_{n+1}(Y)}\\
&=\left\|F^{n+1}\left(\begin{bmatrix}
  \bar x &0\\
0&x_n
\end{bmatrix}\right)-F^{n+1}\left(\begin{bmatrix}
  \bar y &0\\
0&x_n
\end{bmatrix}\right)\right\|_{\M_{n+1}(Y)}\\
&\leq \omega(\|\bar x-\bar y\|_{\M_n(X)})
\end{align*}
for all $n\in\N$ and all $\bar x ,\bar y\in n\cdot B_{\M_n(X)}$. 
Similar computations show that \[\|g^n_n(\bar x)-g^n_n(\bar y)\|_{\M_n(Y)}\geq \rho(\|\bar x-\bar y\|_{\M_n(X)})\] for all $n\in\N$ and all $\bar x ,\bar y\in n\cdot B_{\M_n(X)}$. So, $(g^n\colon n\cdot B_{X}\to Y)_n$ is an almost complete coarse (resp. coarse Lipschitzly) embedding of bounded subsets. 
\end{proof}

We now revisit the method of N. Kalton presented in Section \ref{SectionSomeEmb}. In order to obtain embeddings coming from spherical amplifications, we need to obtain maps $\M_n(X)\to \M_n(Y)$, i.e., maps defined on   whole spaces and not only on bounded sets. For that, the following, which is a version of \cite[Lemma 7.4]{Kalton2012MathAnnalen}, serves as a gluing method. 

\begin{lemma}\label{LemmaKaltonGluing}
Let $X$ and $Y$ be operator spaces, $k\in\N$, $K>0$, and let $(f^t)_{t\geq 0}$ be a family in $\cH(X,Y)$ such that $\|f^t_k\|^{e^{-2t}}\leq K$ for all $t\geq 0$, and 
\[\|f^t_k-f^s_k\|\leq K|t-s|\text{ for all } t,s\geq 0.\]
Then, the map $F\colon \M_k(X)\to \M_k(Y)$ given by 
\[F(\bar x)=\left\{\begin{array}{ll}
0,& \text{ if } \bar x=0,\\
\vphantom{i}[f^0(x_{ij})] ,& \text{ if }\|\bar x\|\leq 1,\\ 
\vphantom{i}[f^{\log\|\bar x\|}(x_{ij})],& \text{ if }\|\bar x\|> 1,
\end{array}\right.  \]
for all $\bar x=[x_{ij}]\in \M_k(X)$, is a coarse spherical  amplification. Precisely, $F$ satisfies $\|F(\bar x)-F(\bar z)\|\leq 2K\|\bar x-\bar z\|+K$ for all $\bar x,\bar z\in \M_k(X)$. 
\end{lemma}

\begin{proof}
It is clear from $F$'s definition that $F$ is a spherical amplification, so we only prove the last statement. For convenience, let $f^t=f^0$ for all $t<0$. Fix $\bar x=[x_{ij}],\bar z=[z_{ij}]\in \M_k(X)$ and suppose $\|\bar x\|\geq \|\bar z\|\geq 0$. Since $\|F(\bar x)\|\leq K\|\bar x\|$, we can assume that   $\|\bar z\|> 0$. If $\|\bar x\|\leq 1$, we have
\[ \|[f^0( x_{ij})-f^0(z_{ij})]\|
\leq K\max\{\|\bar x - \bar z\|, \|\bar x \|\}.\]
If $\|\bar x\|>1$,   we have that
\[ 
\|[f^{\log\|\bar x\|} ( x_{ij})-f^{\log\|\bar x\|}(z_{ij})]\|\leq K\max\{\|\bar x -\bar z \|, \|\bar x \|^{-1}\}.
\]
Combining both inequalities above, we have that 
\begin{align*}
\|[f^{\log\|\bar x\|} ( x_{ij})-f^{\log\|\bar x\|}(z_{ij})]\|&\leq K\max\{\|\bar x -\bar z \|, \min\{\|\bar x\|,\|\bar x \|^{-1}\}\}\\
&\leq  K\|\bar x -\bar z \|+K.
\end{align*}

If $\|\bar z\|\leq 1$ and $\|\bar x\|>1$, then 
\[
\|[f^{0} ( z_{ij})-f^{\log\|\bar x\|}(z_{ij})]\|\leq K  \log\ \|\bar x\| \leq K\|\bar x\|-K\leq K\|\bar x-\bar z\|,
\]
and if $\|\bar z\|>1$, we have
\[
\|[f^{\log\|\bar z\|} ( z_{ij})-f^{\log\|\bar x\|}(z_{ij})]\|\leq K\|\bar z\|\log\frac{\|\bar x\|}{\|\bar z\|}\leq K\|\bar x-\bar z\|.
\]
Therefore, the inequalities above together imply that   $\|F(\bar x)-F(\bar z)\|\leq 2K\|\bar x-\bar z\|+K$.
\end{proof}

\begin{lemma}\label{LemmaSectionPlusKaltonGluing}
Let $X$ and $Y$ be operator spaces, $Q\colon Y\to X$ be a  Banach quotient map, and $k\in\N$. Then, considering $\tilde Q\colon \cZ(Q)\to X$ as defined in Section \ref{SectionSomeEmb}, we have that  $\tilde Q_k$ admits a  coarse section $F\colon \M_k(X)\to \M_k(\cZ(Q))$ which is a spherical amplification. Moreover, $F$ satisfies $\|F(\bar x)-F(\bar z)\|\leq 2\|\bar x-\bar z\|+1$ for all $\bar x, \bar z\in \M_k(X)$.
\end{lemma}

\begin{proof}
By Lemma \ref{LemmaSections}, there is a sequence $(f^n)_{n=0}^\infty$ in $\cH(X,\cZ(Q))$ of sections of $\tilde Q$ such that $\|f^n_k\|^{e^{-2n}}\leq 1$ for all $n\geq 0$. We extend the sequence $(f^n)_{n=0}^\infty$  to a family $(f^t)_{t\geq 0}$  as follows:  for each $n\geq 0$ and $t\in [n,n+1)$, let 
\[f^t(x)=(n+1-t)f^n(x)+(t-n)f^{n+1}(x)\]
for all $x\in X$. So $(f^t)_{t\geq 0}$ is a family of sections of $\tilde Q$ satisfying  $\|f_k^t\|^{e^{-2t}}\leq 1$ for all $t\geq 0$. Moreover, as $\|f^n_k\|\leq 1$ for all $n\in\N$, it follows that $\|f^t_k-f^s_k\|\leq |t-s|$ for all $t,s\geq 0$.

Let $F\colon \M_k(X)\to \M_k(\cZ(Q))$ be the map given by Lemma \ref{LemmaKaltonGluing} applied to the family  $(f^t)_{t\geq 0}$. So, as each $f^t$ is a section of $\tilde Q$, it is clear from its  formula  that $F$  is  a section of $\tilde Q_k$. Moreover, $F$ satisfies all desired properties by Lemma \ref{LemmaKaltonGluing}.
\end{proof}

Theorem \ref{ThmKaltonMethod} has the following version for   spherical embeddings:

\begin{theorem}\label{ThmKaltonMethodSpherical}
Let $X$ and $Y$ be operator spaces,  $Q\colon Y\to X$ be a completely bounded  Banach quotient map, and $\tilde Q\colon \cZ(Q)\to X$ be as above.   Then  $X$   spherically-completely coarse Lipschitzly embeds into  $\cZ(Q)$. 
\end{theorem}

\begin{proof}
Let $Z=\ker(\tilde Q)$. For each $k\in\N$, let $F^{k}\colon \M_{k}(X)\to \M_{k}(\cZ(Q))$ be the coarse spherical amplification section of $\tilde Q_{k}\colon  \M_{k}(\cZ(Q))\to \M_{k}(X)$ given by Lemma \ref{LemmaSectionPlusKaltonGluing}. So, $(F^k)_k$ is equi-coarse.

For each $k\in\N$,  define maps $G^k\colon \M_k(\cZ(Q))\to \M_k(X\oplus Z)$ and $T^k\colon \M_k(X\oplus Z)\to \M_k(\cZ(Q))$ by letting \[G^k(\bar y)=(\tilde Q_k(\bar y), \bar y-F^k(\tilde Q_k(\bar y)))\text{ and }T^k(\bar x,\bar z)=\bar z+F^k(\bar x)\] for all $\bar y\in \M_k(\cZ(Q))$ and all $(\bar x,\bar z)\in \M_k(X\oplus Z)$. Since $(F^k)_k$ is equi-coarse and $\tilde Q$ is completely bounded,   both $(G^k)_k$  and $(F^k)_k$ are   sequences of  equi-coarse maps. Moreover, it is straightforward to check that $G^k$ and $T^k$ are inverses of each other. Therefore, both $(G^k)_k$ and $(T^k)_k$ are equi-coarse embeddings (equivalences even). Since $F^k=T^k\restriction_{ \M_k(X)}$, this shows that $(F^k)_k$ is a sequence of equi-coarse embeddings. Since each $F^k$ is a spherical amplification, $(F^k)_k$ is a  spherically-complete coarse Lipschitz embedding. 
\end{proof}

Notice that while Theorem \ref{ThmKaltonMethod} is a result on the existence of a nonlinear kind of \emph{equivalence}, Theorem \ref{ThmKaltonMethodSpherical} only gives us \emph{embeddings}. Indeed, although the maps $(F^k )_k$ are spherical amplifications, the maps $(G^k)_k$ and $(T_k)_k$ need not be.  

The   next corollary  is the spherical version of Corollary \ref{CorKaltonMethodMIN} and its proof follows completely analogously but with Theorem \ref{ThmKaltonMethodSpherical} replacing Theorem \ref{ThmKaltonMethod}.

 \begin{corollary}\label{CorKaltonMethod2}
 Let $Q\colon \MIN(\ell_1)\to\MIN( c_0)$ be a completely bounded  map which is also a Banach quotient, and let $\tilde Q\colon \cZ(Q)\to X$ be as above.  Then  
\begin{enumerate}
\item $\MIN( c_0) $ does not  isomorphically embed  into  $\cZ(Q)$, and 
\item   the bounded subsets of $\cZ(Q)$ and  $\MIN( c_0)\oplus \ker(\tilde Q)$  are  almost completely coarse Lispchitzly equivalent.\qed
\end{enumerate} 
 \end{corollary}

\begin{proof}[Proof of Theorem \ref{ThmSepSphericalEmb}]
This is simply Corollary \ref{CorKaltonMethod2}.
\end{proof}

\begin{acknowledgments}
The author would like to thank Javier  Alejandro  Ch\'{a}vez-Dom\'{i}nguez, Timur Oikhberg, Gilles Pisier, and Thomas Sinclair for very enlightening conversations about operator spaces. In particular, the author is extremely grateful to Timur Oikhberg for his help with  Proposition \ref{PropCompEmbSum} and to Gilles Pisier for  suggesting to look at the quotient $Q\colon L_\infty[0,1]\to R\cap C$ for Corollary \ref{CorKaltonMethod}.
\end{acknowledgments}

\end{document}